\documentclass{amsart}

\usepackage[english]{babel} 

\usepackage{amsmath}
\usepackage{amsthm}
\usepackage{amsfonts}
\usepackage{amssymb}
\usepackage{thmtools}
\usepackage{tikz-cd}
\usepackage{mathtools}
\usepackage{enumitem}
\usepackage{dsfont}
\usepackage{hyperref}
\usepackage{cleveref}

\newcommand{\N}{\mathbb{N}}
\newcommand{\Z}{\mathbb{Z}}

\newcommand{\R}{\mathbb{R}}
\newcommand{\Np}{\mathbb{N}_{>0}}

\newcommand{\MV}{{\mathsf{MV}}} %MV-algebras%
\newcommand{\MVM}{\mathsf{MVM}} %MVM-algebras

\newcommand{\ULG}{{\mathsf{u} \ell \mathsf{G}}} %unital l-groups
\newcommand{\ULM}{{\mathsf{u} \ell \mathsf{M}}} %unital l-monoids
\newcommand{\ULMP}{{\mathsf{u} \ell \mathsf{M}^+ }} %positive-unital l-monoids

\newcommand{\cG}{\overline{\mathsf{G}}} %divisible Archimedean metrically complete Abelian u$\ell$-groups

\newcommand{\CH}{\mathsf{CompHaus}}

\newcommand{\G}{\tilde{\Gamma}}
\newcommand{\X}{\tilde{\Xi}}
\newcommand{\T}{\mathds{T}}
\newcommand{\GS}{\mathds{G}}
\newcommand{\U}{\mathds{U}}
\newcommand{\Id}{1}

\newcommand{\mv}{MV-algebra}
\newcommand{\mvs}{MV-algebras}
\newcommand{\amv}{an MV-algebra}

\newcommand{\mvm}{MV-monoidal algebra}
\newcommand{\mvms}{MV-monoidal algebras}
\newcommand{\amvm}{an MV-monoidal algebra}

\newcommand{\ulg}{unital Abelian $\ell$-group}
\newcommand{\ulgs}{unital Abelian $\ell$-groups}
\newcommand{\aulg}{a unital Abelian $\ell$-group}

\newcommand{\extulg}{unital Abelian lattice-ordered group}
\newcommand{\extulgs}{unital Abelian lattice-ordered groups}

\newcommand{\lm}{commutative $\ell$-monoid}

\newcommand{\alm}{a commutative $\ell$-monoid}

\newcommand{\extlm}{commutative lattice-ordered monoid}

\newcommand{\ulm}{unital commutative $\ell$-monoid}
\newcommand{\ulms}{unital commutative $\ell$-monoids}
\newcommand{\aulm}{a unital commutative $\ell$-monoid}

\newcommand{\extulm}{unital commutative lattice-ordered monoid}
\newcommand{\extulms}{unital commutative lattice-ordered monoids}
\newcommand{\extaulm}{a unital commutative lattice-ordered monoid}

\newcommand{\pulm}{{positive-unital commutative $\ell$-monoid}}
\newcommand{\pulms}{{positive-unital commutative $\ell$-monoids}}
\newcommand{\apulm}{{a positive-unital commutative $\ell$-monoid}}

\renewcommand{\leq}{\leqslant}
\renewcommand{\geq}{\geqslant}
\newcommand{\seq}{\subseteq}
\newcommand{\df}{\coloneqq}

\newcommand{\0}{\mathbf{0}}
\newcommand{\1}{\mathbf{1}}

\newcommand{\eps}{\varepsilon}

\newcommand{\cat}[1]{\mathsf{#1}}

\newcommand{\gs}[1]{\mathbf{#1}} 

\newcommand{\by}[1]{\text{(#1)}} % for citations in align environment

\declaretheorem[name = Theorem, refname = {Theorem,Theorems}, Refname = {Theorem,Theorems}, numberwithin = section,]{theorem}
\declaretheorem[name = Proposition, refname = {Proposition,Propositions}, Refname = {Proposition,Propositions}, sibling = theorem,]{proposition}
\declaretheorem[name = Lemma, refname = {Lemma,Lemmas}, Refname = {Lemma,Lemmas}, sibling = theorem,]{lemma}
\declaretheorem[name = Corollary, refname = {Corollary,Corollaries}, Refname = {Corollary,Corollaries}, sibling = theorem,]{corollary}

\declaretheorem[name = Definition, refname = {Definition,Definitions}, Refname = {Definition,Definitions}, sibling = theorem, style = definition,]{definition}
\declaretheorem[name = Notation, refname = {Notation,Notations}, Refname = {Notation,Notations}, sibling = theorem, style = definition,]{notation}
\declaretheorem[name = Example, refname = {Example,Examples}, Refname = {Example,Examples}, sibling = theorem, style = definition,]{example}
\declaretheorem[name = Remark, refname = {Remark,Remarks}, Refname = {Remark,Remarks}, sibling = theorem, style = definition,]{remark}

\crefname{section}{Section}{Sections}
\Crefname{section}{Section}{Sections}

\crefname{axiom}{Axiom}{Axioms}
\Crefname{axiom}{Axiom}{Axioms}

\crefname{appendix}{Appendix}{Appendices}
\Crefname{appendix}{Appendix}{Appendices}

\begin{document}

\title{Equivalence \`a la Mundici for commutative lattice-ordered monoids}

\author[M.\ Abbadini]{Marco Abbadini}
\address{Department of Mathematics, University of Salerno,
Via Giovanni Paolo II, 132 Fisciano (SA), Italy}
\email{mabbadini@unisa.it}


\begin{abstract}
	We provide a generalization of Mundici's equivalence between {\extulgs} and {\mvs}: the category of \emph{\extulms} is equivalent to the category of \emph{\mvms}.
	Roughly speaking, the structures we call {\extulms} are {\extulgs} without the unary operation $x \mapsto -x$.
	The primitive operations are $+$, $\lor$, $\land$, $0$, $1$, $-1$.
	A prime example of these structures is $\mathbb{R}$, with the obvious interpretation of the operations.
	Analogously, {\mvms} are {\mvs} without the negation $x \mapsto \lnot x$.
	The primitive operations are $\oplus$, $\odot$, $\lor$, $\land$, $0$, $1$.
	A motivating example of {\mvm} is the negation-free reduct of the standard {\mv} $[0, 1]\subseteq \mathbb{R}$.
	We obtain the original Mundici's equivalence as a corollary of our main result.
\end{abstract}

\maketitle

%%%%%%%%%%%%%%%%%%%%%%%%%%%%%%%%%%% SECTION %%%%%%%%%%%%%%%%%%%%%%%%%%%%%%%%%%%%

\section{Introduction}

In \cite{Mundici}, Mundici proved that the category of {\extulgs} (\ulgs, for short) is equivalent to the category of {\mvs}.
In \cref{t:G is equivalence}, our main result, we establish the following generalization: \emph{The category of {\extulms} is equivalent to the category of {\mvms}}.

Roughly speaking, {\extulms} (\emph{{\ulms}}, for short) are {\ulgs} without the unary operation $x \mapsto -x$, whereas {\mvms} are {\mvs} without the negation $x \mapsto \lnot x$ (precise definitions will be given in \cref{s:definition}).
The operations of {\ulms} are $+$, $\lor$, $\land$, $0$, $1$, $-1$, whereas the operations of {\mvms} are $\oplus$, $\odot$, $\lor$, $\land$, $0$, $1$.
A motivating example of {\ulm} is $\R$, with the obvious interpretation of the operations, whereas a motivating example of {\mvm} is the negation-free reduct of the standard {\mv} $[0, 1]$.
Furthermore, for every topological space $X$ equipped with a preorder, the set of bounded continuous order-preserving functions from $X$ to $\R$ is an example of a {\ulm}, whereas the set of continuous order-preserving functions from $X$ to $[0, 1]$ is an example of an {\mvm}.
The author's interest for {\ulms} originated from these last examples, as we now illustrate in some detail.

Given a compact Hausdorff space $X$, the set $C(X, \R)$ of continuous functions from $X$ to $\R$ is a divisible Archimedean {\ulg}, complete in the uniform metric.
In fact, we have a duality between the category $\CH$ of compact Hausdorff spaces and continuous maps and the category $\cG$ of divisible Archimedean metrically complete {\ulgs} (see \cite{Yosida,Stone}).	Similarly, we may consider on the set $C(X,[0, 1])$ of continuous functions from $X$ to $[0, 1]$ pointwise-defined operations inherited from $[0, 1]$; for example, the operations of {\mvs}.
Developing this idea, one can show that $\CH$ is dually equivalent to a variety $\Delta$ of (infinitary) algebras (see \cite{Duskin,Isbell,MarraReggio}).
These algebras can be thought of as {\mvs} with an additional operation of countably infinite arity satisfying some additional axioms.	In fact, we have an equivalence between $\cG$ and $\Delta$, which is essentially a restriction of the equivalence between {\ulgs} and {\mvs}.

A \emph{compact ordered space} is a compact space $X$ endowed with a partial order $ \leq $ on $X$ so that the set $\{\,(x, y) \in X \times X \mid x \leq y\,\}$ is closed in $X \times X$ with respect to the product topology. This notion was introduced by Nachbin \cite{Nachbin}.
If we replace compact Hausdorff spaces by compact ordered spaces in the aforementioned discussion involving $\CH$, $\Delta$ and $\cG$, then we may accordingly replace Mundici's equivalence with our \cref{t:G is equivalence}.
Given a compact ordered space $X$, let us consider the set $C_{\scriptscriptstyle \leq }(X,[0, 1])$ of continuous order-preserving functions from $X$ to $[0, 1]$:
we can endow $C_{\scriptscriptstyle \leq }(X,[0, 1])$ with pointwise-defined operations $\oplus$, $\odot$, $\lor$, $\land$, $0$, $1$ (which are the operations of {\mvms}).
Pursuing a similar idea, in \cite{HofmannShort} it was proved that the category of compact ordered spaces and continuous order-preserving maps is dually equivalent to a quasi-variety of infinitary algebras (\cite{Abbadini,AbbadiniReggio2020} show that this quasi-variety actually is a variety).
However, the operations are somewhat unwieldy, and one might want to investigate the set $C_{\scriptscriptstyle \leq }(X, \R)$ of continuous order-preserving real-valued functions, instead.
In fact, $C_{\scriptscriptstyle \leq }(X, \R)$ is a {\ulm}.
The main motivation of this paper is to make the connection between $C_{\scriptscriptstyle \leq }(X, \R)$ ({\ulms}) and $C_{\scriptscriptstyle \leq }(X,[0, 1])$ ({\mvms}) explicit.

There are both pros and cons in working with {\ulms} or {\mvms}.
On one hand, it is easier to work with the axioms of {\ulms} rather than those of {\mvms}.
On the other hand, the category of {\mvms} is a variety of finitary algebras axiomatized by a finite number of equations, so the tools of universal algebra apply.
The equivalence established here allows to transfer the pros of one category to the other one.

Our result specializes to Mundici's equivalence between {\ulgs} and {\mvs} (\cref{s:restriction}).
We remark that, in contrast to the proof of Mundici's equivalence in \cite{Mundici}, we do not use the axiom of choice to prove the equivalence between {\ulms} and {\mvms}.

We sketch the proof of our main result, \cref{t:G is equivalence}.
In order to obtain an equivalence 
\[
	\begin{tikzcd}
		\ULM \arrow[yshift = .45ex]{r}{\G}	& \MVM \arrow[yshift = -.45ex]{l}{\X}
	\end{tikzcd}
\]
between the category $\ULM$ of {\ulms} and the category $\MVM$ of {\mvms}, we show that we have two equivalences
\[
	\begin{tikzcd}
		\ULM \arrow[yshift = .8ex, bend left]{rr}{\G} \arrow[yshift = .45ex]{r}{(-)^+ }
		& \ULMP \arrow[yshift = .45ex]{r}{\U} \arrow[yshift = -.45ex]{l}{\T}
		& \MVM \arrow[yshift = -.45ex]{l}{\GS} \arrow[yshift = -.8ex, bend left]{ll}{\X}.
	\end{tikzcd}
\]
Here $\ULMP$ is the category of `{\pulms}' (\cref{d:positive-unital}), which are the positive cones of {\ulms}.
The functor $\G$ maps a {\ulm} $M$ to its `unit interval' $\G(M)$ (\cref{s:Gamma}).
We construct a quasi-inverse in two steps.
As a first step, given an {\mvm} $A$, we define the set $\GS(A)$ of `good sequences in $A$' (\cref{s:good.def}), and we equip this set with the structure of a {\pulm} (\cref{s:operations on Good}).
As a second step, we consider translations of the elements of $\GS(A)$ by negative integers; in this way we obtain a {\ulm} $\T\GS(A)$, where $\T\colon \ULMP\to \ULM$ is a functor (\cref{s:positive cone}).
To show that the composition of these two steps provides a quasi-inverse of $\G$, we write $\G$ as the composite of two functors $(-)^+ $ and $\U$.
The functor $(-)^+ $ associates to $M$ its `positive cone' $M^+ $; the functor $\U$ associates to $M^+ $ its unit interval.
We will show that $(-)^+ $ and $\T$ are quasi-inverses (\cref{s:positive cone}), and that $\U$ and $\GS$ are quasi-inverses (\cref{s:MVM and positive-unital}); from this, it follows that $\G$ and $\X \df \T\GS$ are quasi-inverses, and hence the categories of {\ulms} and {\mvms} are equivalent (\cref{t:G is equivalence}).

By the time of publication of the present paper, the main result (\cref{t:G is equivalence}) has appeared also in the author's Ph.D.\ thesis \cite[Chapter~4]{Abbadini2021}.
However, the proofs are different.
In \cite{Abbadini2021}, the author uses Birkhoff’s subdirect representation theorem, which simplifies the arguments but relies on the axiom of choice. %, in contrast with the choice-free proof in the present paper.
Moreover, while in this paper we construct a quasi-inverse of $\G$ as the composite of two functors, in \cite{Abbadini2021} a one-step construction is adopted. %: {\amvm} $A$ is mapped to an algebra whose underlying set is the set of so-called `good $\Z$-sequences in $A$', which are functions from $\Z$ to $A$ satisfying certain properties.

In \cref{s:sub-irr_tot-ord} we show that subdirectly irreducible {\mvms} are totally ordered, and in \cref{s:good-pairs_in_sub-irr} we show that every good sequence in a subdirectly irreducible {\mvm} is of the form $(1, \dots, 1, x, 0, 0, \dots)$.
Even if we do not use these last two results, we have included them because they seem of interest.

%%%%%%%%%%%%%%%%%%%%%%%%%%%%%%%%%%% SECTION %%%%%%%%%%%%%%%%%%%%%%%%%%%%%%%%%%%%

\section{The algebras} \label{s:definition}

%================================= SUBSECTION =================================%

\subsection{Unital commutative lattice-ordered monoids}

The set $\R$, endowed with the binary operations $ + $ (addition), $ \lor $ (maximum), $ \land $ (minimum), and the constants $0$, $1$ and $-1$ is a prototypical example of \emph{\extaulm}.

\begin{definition} \label{d:lm}
	A \emph{\extlm} (shortened as \emph{\lm}) is an algebra $\langle M; +, \lor, \land, 0 \rangle$ (arities $2$, $2$, $2$, $0$) with the following properties.
	\begin{enumerate}[label=\textup{(M\arabic*)}, ref = M\arabic*]
		\item	\label[axiom]{ax:M1} $\langle M; \lor, \land \rangle$ is a distributive lattice.
		\item	\label[axiom]{ax:M2} $\langle M; +, 0 \rangle$ is a commutative monoid.
		\item	\label[axiom]{ax:M3} $ + $ distributes over $ \lor $ and $ \land $.
	\end{enumerate}
\end{definition}

\begin{definition} \label{d:ulms}
	A \emph{\extulm} (\emph{\ulm}, for short) is an algebra $\langle M; +, \lor, \land, 0, 1, -1 \rangle$ (arities $2$, $2$, $2$, $0$, $0$, $0$) with the following properties. %such that $\langle M; \lor, \land, 0 \rangle$ is {\alm}, and the following properties are satisfied.
	\begin{enumerate}[label=\textup{(U\arabic*)}, ref = U\arabic*, start = 0]
		\item \label[axiom]{ax:U0}
				$\langle M; +, \lor, \land, 0 \rangle$ is {\alm}.
		\item	\label[axiom]{ax:U1}
			$-1 + 1 = 0$.
		\item \label[axiom]{ax:U2}
			$0 \leq 1$.
		\item \label[axiom]{ax:U3}
			For every $x \in M$ there exists $n \in \N$ such that
			\[
				\underbrace{(-1) + \dots + (-1)}_{n \ \text{times}} \leq x \leq \underbrace{1 + \dots + 1}_{n \ \text{times}}.
			\]
	\end{enumerate}
	The element $1$ is called the \emph{positive unit}, and the element $-1$ is called the \emph{negative unit}.
\end{definition}

We warn the reader that some authors do not assume the lattice to be distributive, nor that $ + $ distributes over both $\land$ and $\lor$.

We denote with $\ULM$ the category of {\ulms} with homomorphisms.
We write $z - 1$ for $z + (-1)$.
Given $n \in \N$, we write $n$ for $\underbrace{1 + \dots + 1}_{n \ \text{times}}$ and $-n$ for $\underbrace{(-1) + \dots + (-1)}_{n \ \text{times}}$.

\begin{remark} \label{r:te}
	Given {\aulm} $\langle M; +, \lor, \land, 0, 1, -1 \rangle$, we define the operation $x \cdot y \df x - 1 + y$.
	This operation does not coincide on $\R$ with the usual multiplication.
	However, we still use this notation because the equations $x \cdot 1 = x = 1 \cdot x$ hold.
	In fact, {\ulms} admit a term-equivalent description in the signature $\{+, \cdot, \lor, \land, 0, 1\}$, from which the constant $-1$ can be recast as $0 \cdot 0$.
	In this signature, \cref{ax:U0,ax:U1} are equivalent to:
	\begin{enumerate}[label=\textup{(E\arabic*)}, ref = E\arabic*]
		\item \label[axiom]{ax:E1}$\langle M; \lor, \land \rangle$ is a distributive lattice.		
		\item	\label[axiom]{ax:E2}
				$\langle M; +, 0 \rangle$ and $\langle M; \cdot, 1 \rangle$ are commutative monoids.
		\item	\label[axiom]{ax:E3}
				Both the operations $ + $ and $\cdot$ distribute over both $ \lor $ and $ \land $.
		\item \label[axiom]{ax:E4}
				$(x \cdot y) + z = x \cdot (y + z)$.
		\item \label[axiom]{ax:E5}
				$(x + y) \cdot z = x + (y \cdot z)$.
	\end{enumerate}
	(Note that \cref{ax:E5,ax:E4} are equivalent, given the commutativity of $ + $ and $\cdot$.)
	The addition of \cref{ax:U2} equals the addition of the following axiom.
	\begin{enumerate}[label=\textup{(E\arabic*)}, ref = E\arabic*, start = 5]
		\item \label[axiom]{ax:E6}
				$0 \leq 1$.
	\end{enumerate}
	The addition of \cref{ax:U3} equals the addition of the following axiom.
	\begin{enumerate}[label=\textup{(E\arabic*)}, ref = E\arabic*, start = 6]
		\item \label[axiom]{ax:E7}
				For every $x \in M$ there exists $n \in \N$ such that 
				\[
					\underbrace{0 \cdot \dots \cdot 0}_{n \ \text{times}} \leq x \leq \underbrace{1 + \dots + 1}_{n \ \text{times}}.
				\]
	\end{enumerate}
	The class of algebras satisfying \cref{ax:E1,ax:E2,ax:E3,ax:E4,ax:A5,ax:E6,ax:E7} is term-equivalent to the class of {\ulms}.
	One interesting thing of \cref{ax:E1,ax:E2,ax:E3,ax:E4,ax:E6,ax:E7} is that their symmetries resemble the ones in the definition of {\mvms} below.
	We will use \cref{ax:E1,ax:E2,ax:E3,ax:E4,ax:E6,ax:E7} to explain the axioms of {\mvms} below; besides this usage, we will stick to \cref{ax:U0,ax:U1,ax:U2,ax:U3} throughout the paper.
\end{remark}

\begin{example}
For every topological space $X$ equipped with a preorder, the set of bounded continuous order-preserving functions from $X$ to $\R$ is a {\ulm}.
\end{example}

%================================= SUBSECTION =================================%

\subsection{MV-monoidal algebras}

In the following, we define the variety $\MVM$ of {\mvms}, which are finitary algebras axiomatised by a finite number of equations.
Our main result is that the categories $\ULM$ and $\MVM$ are equivalent.
Without giving the details now, we anticipate the fact that the equivalence is given by the functor $\G \colon \ULM \to \MVM$ that maps a {\ulm} $M$ to the set $\{\,x \in M \mid 0 \leq x \leq 1\,\}$, endowed with the operations $\oplus$, $\odot$, $ \land $, $ \lor $, $0$ and $1$, where $ \land $, $ \lor $, $0$ and $1$ are defined by restriction, and $\oplus$ and $\odot$ are defined by $x \oplus y \df (x + y) \land 1$ and $x \odot y \df (x + y - 1) \lor 0$.

On $[0, 1]$, consider the elements $0$ and $1$ and the operations $x \lor y \df \max \{x, y\}$, $x \land y \df \min \{x, y\}$, $x \oplus y \df \min \{x + y, 1\}$, and $x \odot y \df \max \{x + y - 1, 0\}$.
This gives a prime example of what we call {\amvm}.

\begin{definition} \label{d:MVM}
	An \emph{\mvm} is an algebra $\langle A; \oplus, \odot, \lor, \land, 0, 1 \rangle$ (arities $2$, $2$, $2$, $2$, $0$, $0$) satisfying the following equational axioms.
	\begin{enumerate}[label=\textup{(A\arabic*)}, ref = A\arabic*]
		\item	\label[axiom]{ax:A1}
				$\langle A; \lor, \land \rangle$ is a distributive lattice.		
		\item	\label[axiom]{ax:A2}
				$\langle A; \oplus, 0 \rangle$ and $\langle A; \odot, 1 \rangle$ are commutative monoids.
		\item	\label[axiom]{ax:A3}
				Both the operations $\oplus$ and $\odot$ distribute over both $ \lor $ and $ \land $.
		\item	\label[axiom]{ax:A4}
				$(x \oplus y) \odot ((x \odot y) \oplus z) = (x \odot (y \oplus z)) \oplus (y \odot z)$. 
		\item \label[axiom]{ax:A5}
				$(x \odot y) \oplus ((x \oplus y) \odot z) = (x \oplus (y \odot z)) \odot (y \oplus z)$.
		\item	\label[axiom]{ax:A6}
				$(x \odot y) \oplus z = ((x \oplus y) \odot ((x \odot y) \oplus z)) \lor z$.
		\item	\label[axiom]{ax:A7}
				$(x \oplus y) \odot z = ((x \odot y) \oplus ((x \oplus y) \odot z)) \land z$.
	\end{enumerate}
\end{definition}

Before commenting on the axioms, we remark that \cref{ax:A4,ax:A5} are equivalent, given the commutativity of $\oplus$ and $\odot$.
We have included both so to make it clear that, if $\langle A; \oplus, \odot, \lor, \land, 0, 1 \rangle$ is {\amvm}, then also the `dual' algebra $\langle A; \odot, \oplus, \land, \lor, 1, 0 \rangle$ is {\amvm}.

\Cref{ax:A1,ax:A2,ax:A3} coincide with \cref{ax:E1,ax:E2,ax:E3} in \cref{r:te}.
So, in a sense, the difference between {\mvms} and {\ulms} lies in the difference bewteen the conjunction of \cref{ax:A4,ax:A5,ax:A6,ax:A7} and the conjunction of \cref{ax:E4,ax:E6,ax:E7}.
We mention here that $0$ and $1$ are bounds of the underlying lattice of {\amvm} (\cref{l:bounded-lattice}); this fact is not completely obvious, given that the proof makes use of almost all the axioms of {\mvms}.

\Cref{ax:A4} is a sort of associativity, which resembles \cref{ax:E4}, i.e.\ $(x \cdot y) + z = x \cdot (y + z)$.
In particular, one verifies that the interpretation on $[0, 1]$ of both the left-hand and right-hand side of \cref{ax:A4} equals
\begin{equation} \label{e:xyz}
	((x + y + z - 1) \lor 0) \land 1.
\end{equation}
Notice that the element $x + y + z - 1$ appearing in \eqref{e:xyz} coincides, using the definition of $\cdot$ from \cref{r:te}, with the interpretation on $\R$ of $(x \cdot y) + z$ and $x \cdot (y + z)$.
In fact, \cref{ax:A4} is essentially the condition $(x \cdot y) + z = x \cdot (y + z)$ expressed at the unital level, i.e.: 
\begin{equation} \label{e:xyz-01}
	(((x \cdot y) + z) \lor 0) \land 1 = ((x \cdot (y + z)) \lor 0) \land 1.
\end{equation}
Indeed, the presence of the term $x \cdot y$ in the left-hand side of \eqref{e:xyz-01} corresponds to the presence of the terms $x \oplus y$ and $x \odot y$ in the left-hand side of \cref{ax:A4}, and the presence of the term $y + z$ in the right-hand side of \eqref{e:xyz-01} corresponds to the presence of the terms $y \oplus z$ and $y \odot z$ in the right-hand side of \cref{ax:A4}.

Analogously, \cref{ax:A5} corresponds to \cref{ax:E5}, i.e.\ $(x + y) \cdot z = x + (y \cdot z)$.

\Cref{ax:A6} expresses how the term $(x \odot y) \oplus z$ differs from its non-truncated version $(x \cdot y) + z$: essentially, the axiom can be read as 
\[
	(x \odot y) \oplus z = ((x \cdot y) + z) \lor z.
\]
Analogously, \cref{ax:A7} can be read as
\[
	(x \oplus y) \odot z = ((x + y) \cdot z) \land z.
\]

We remark that {\mvms} form a variety of algebras whose primitive operations are finitely many and of finite arity, and which is axiomatised by a finite number of equations.
We let $\MVM$ denote the category of {\mvms} with homomorphisms.

\begin{remark} \label{r:bdl are MVM}
	Bounded distributive lattices form a subvariety of the variety of {\mvms}, obtained by adding the axioms $x \oplus y = x \lor y$ and $x \odot y = x \land y$.
\end{remark}

%%%%%%%%%%%%%%%%%%%%%%%%%%%%%%%%%%% SECTION %%%%%%%%%%%%%%%%%%%%%%%%%%%%%%%%%%%%

\section{The unit interval functor} \label{s:Gamma}

In this section we define a functor $\G$ from the category of {\ulms} to the category of {\mvms}; the main goal of the paper is to show that $\G$ is an equivalence.
For a {\ulm} $M$, we set $\G(M) \df \{\,x \in M \mid 0 \leq x \leq 1\,\}$.
We are going to endow $\G(M)$ with a structure of {\amvm}.
Clearly, $0, 1 \in \G(M)$.
Moreover, we define $ \lor $ and $ \land $ on $\G(M)$ by restriction.
Finally, for $x, y \in \G(M)$, we set 
\begin{align*}
	x \oplus y	& \df (x + y) \land 1 \\
\intertext{and}
	x \odot y	& \df (x + y-1) \lor 0,
\end{align*}
To see that $\oplus$ and $\odot$ are internal operations on $\G(M)$, we make use of the following.

\begin{lemma} \label{l:sum positive}
	Let $M$ be {\alm}.
	For all $x, y, x', y' \in M$ such that $x \leq x'$ and $y \leq y'$, we have $x + y \leq x' + y'$.
\end{lemma}

\begin{proof}
	Since $ + $ distributes over $ \lor $, we have $(x + y) \lor (x + y') = x + (y \lor y') = x + y' $.
	Therefore, $x + y \leq x + y'$; analogously, $x + y' \leq x' + y'$.
	Hence, $x + y \leq x + y' \leq x' + y'$.
\end{proof}

\noindent By \cref{l:sum positive}, $\oplus$ and $\odot$ are internal operations on $\G(M)$: indeed, the condition $x \oplus y \in \G(M)$ holds because $x + y \geq 0 + 0 = 0$, and the condition $x \odot y \in \G(M)$ holds because $x + y - 1 \leq 1 + 1 - 1 = 1$.

Our next goal---met in \cref{p:Gamma is MVM} below---is to show that $\G(M)$ is an {\mvm}.
We need some lemmas.

\begin{lemma} \label{l:plusdot in mon}
	Let $M$ be a {\ulm}, and let $x, y, z\in \G(M)$.
	Then
	\begin{align*}
		(x \odot y) \oplus z	& = ((x + y + z - 1) \lor z) \land 1, \\
	\intertext{and}
		(x \oplus y) \odot z	& = ((x + y + z - 1) \land z) \lor 0.
	\end{align*}
\end{lemma}

\begin{proof}
	We have
	\begin{align*}
		(x \odot y) \oplus z	& = ((x \odot y) + z) \land 1
									&& \by{def.\ of $\oplus$} \\
									& = (((x + y - 1) \lor 0) + z) \land 1
									&& \by{def.\ of $\odot$} \\
									& = ((x + y + z - 1) \lor z) \land 1
									&& \by{$ + $ distr.\ over $ \land $}
	\intertext{and}
		(x \oplus y) \odot z	& = ((x \oplus y) + z - 1) \lor 0
									&& \by{def.\ of $\odot$} \\
									& = (((x + y) \land 1) + z - 1) \lor 0
									&& \by{def.\ of $\oplus$} \\
									& = ((x + y + z - 1) \land z) \lor 0.
									&& \by{$ + $ distr.\ over $ \land $} \qedhere
	\end{align*}
\end{proof}

\begin{lemma} \label{l:land+lor=+}
	For all $x$ and $y$ in {\alm} we have
	\[
		(x \land y) + (x \lor y) = x + y.
	\]
\end{lemma}

\begin{proof}
	We recall the proof, available in \cite{Choud}, of the two inequalities:
	\begin{gather*}
		(x \land y) + (x \lor y) = ((x \land y) + x) \lor ((x \land y) + y) \leq (y + x) \lor (x + y) = x + y;\\
		(x \land y) + (x \lor y) = (x + (x \lor y)) \land (y + (x \lor y)) \geq (x + y) \land (y + x) = x + y. \qedhere
	\end{gather*}
\end{proof}

\begin{lemma}
	\label{l:oplus odot is good}
	Let $M$ be {\aulm}, and let $x, y \in \G(M)$.
	Then
	\[
		(x \oplus y) + (x \odot y) = x + y.
	\]
\end{lemma}

\begin{proof}
	We have
	\begin{align*}
		(x \oplus y) + (x \odot y)	& = ((x + y) \land 1) + ((x + y - 1) \lor 0) && \by{def.\ of $\oplus$ and $\odot$} \\
											& = ((x + y) \land 1) + ((x + y) \lor 1) - 1 && \by{$+$ distr.\ over $\lor$} \\
											& = x + y + 1 - 1										&& \by{\cref{l:land+lor=+}} \\
											& = x + y.												&&
		\qedhere
	\end{align*}
\end{proof}

\begin{lemma} \label{l:sigma in monoid}
	Let $M$ be a {\ulm}.
	For all $x, y, z\in \G(M)$, the elements $(x \oplus y) \odot ((x \odot y) \oplus z)$, $(x \odot y) \oplus ((x \oplus y) \odot z)$, $(x \odot (y \oplus z)) \oplus (y \odot z)$, and $(x \oplus (y \odot z)) \odot (y \oplus z)$ coincide with
	\[
		(x + y + z - 1) \lor 0) \land 1.
	\]
\end{lemma}

\begin{proof}
	We have
	\begin{align*}
		& (x \oplus y) \odot ((x \odot y) \oplus z)								&& \\
		& = (((x \oplus y) + (x \odot y) + z) \land (x \oplus y)) \lor 0 	&& \by{\cref{l:plusdot in mon}} \\
		& = ((x + y + z - 1) \land (x \oplus y)) \lor 0							&& \by{\cref{l:oplus odot is good}} \\
		& = ((x + y + z - 1) \land (x + y) \land 1) \lor 0						&& \by{def.\ of $\oplus$} \\
		& = ((x + y + z - 1) \land 1) \lor 0										&& \by{$x + y + z - 1 \leq x + y$} \\
		& = ((x + y + z - 1) \lor 0) \land 1.										&&				
	\intertext{and}
		& (x \odot (y \oplus z)) \oplus (y \odot z)								&& \\
		& = ((x + (y \oplus z) + (y \odot z)) \lor (y \odot z)) \land 1	&& \by{\cref{l:plusdot in mon}} \\
		& = ((x + y + z - 1) \lor (y \odot z)) \land 1							&& \by{\cref{l:oplus odot is good}} \\
		& = ((x + y + z - 1) \lor (y + z - 1) \lor 0) \land 1					&& \by{def.\ of $\odot$} \\
		& = ((x + y + z - 1) \lor 0) \land 1.										&& \by{$x + y + z - 1 \geq y + z - 1$}
	\end{align*}
	The fact that also $(x \odot y) \oplus ((x \oplus y) \odot z)$ and $(x \oplus (y \odot z)) \odot (y \oplus z)$ coincide with $(x + y + z - 1) \lor 0) \land 1$ follows from the commutativity of $\oplus$ and $\odot$ (which is easily seen to hold) and the commutativity of $ + $.
\end{proof}

\begin{proposition} \label{p:Gamma is MVM}
	Let $M$ be a {\ulm}.
	Then $\G(M)$ is an {\mvm}.
\end{proposition}

\begin{proof}
	\Cref{ax:A1,ax:A2,ax:A3} are obtained by straightforward computations.
	\Cref{ax:A4,ax:A5} hold by \cref{l:sigma in monoid}.
	\Cref{ax:A6,ax:A7} hold by \cref{l:plusdot in mon,l:sigma in monoid}.
\end{proof}

Given a morphism of {\ulms} $f\colon M \to N$, we denote with $\G(f)$ its restriction $\G(f) \colon \G(M) \to \G(N)$.
This establishes a functor
\[
	\G \colon \ULM \to \MVM.
\]
Our main goal is to show that $\G$ is an equivalence of categories.

%%%%%%%%%%%%%%%%%%%%%%%%%%%%%%%%%%% SECTION %%%%%%%%%%%%%%%%%%%%%%%%%%%%%%%%%%%%

\section{Positive cones} \label{s:positive cone}

In \cite[Chapter 2]{Cignoli}, the authors proceed in two steps in order to prove that, for an {\mv} $A$, there exists {\aulg} that envelops $A$.
First, a partially ordered monoid $M_A$ is constructed from $A$.
Then {\aulg} $G_A$ is defined (in a way which is analogous to the definition of $\Z$ from $\N$).
In this paper, we proceed analogously: the role of $A$ is played by {\mvms}, the role of $G_A$ is played by {\ulms}, and the role of $M_A$ is played by what we call \emph{\pulms}.
Roughly speaking, if we think of a {\ulm} as the interval $(-\infty, \infty)$, then an {\mvm} is the interval $[0, 1]$, whereas a {\pulm} is the interval $[0, \infty)$.

In order to prove that $\G$ is an equivalence, we show that $\G$ is the composite of two equivalences 
\[
	\begin{tikzcd}
		\ULM \arrow{r}{(-)^+ } & \ULMP \arrow{r}{\U} & \MVM,
	\end{tikzcd}
\]
where $\ULMP$ is the category---yet to be defined---of \pulms.
The idea is that, for $M \in \ULM$, we have $M^+ \df \{\,x \in M \mid x \geq 0\,\}$, and for $N \in \ULMP$, we have $\U(N) \df \{\,x \in N \mid x \leq 1\,\}$, so that $\U(M^+ ) = \{\,x \in M \mid 0 \leq x \leq 1\,\} = \G(M)$.
In this section, we define the functor $(-)^+ $, and we exhibit a quasi-inverse $\T$.
We remark that one could construct a quasi-inverse functor for $\G$ just in one step: see the author's Ph.D.\ thesis \cite[Chapter~4]{Abbadini2021} for the employment of this approach.

Given a {\ulm} $M$, we set $M^+ \df \{\,x \in M \mid x \geq 0\,\}$.
With the following definition, we aim to capture the structure of $M^+ $ for $M$ {\aulm}.

\begin{definition} \label{d:positive-unital}
	By a \emph{\pulm} we mean an algebra $\langle M; +, \lor, \land, 0, 1, - \ominus 1 \rangle$ (arities $2$, $2$, $2$, $0$, $0$, $1$) such that, for every $x \in M$, the following properties hold.
	\begin{enumerate} [label=\textup{(P\arabic*)}, ref = P\arabic*, start = 0]
		\item \label[axiom]{ax:P0} $\langle M; +, \lor, \land, 0 \rangle$ is a {\lm}.
		\item \label[axiom]{ax:P1} $x \geq 0$.
		\item \label[axiom]{ax:P2} $(x + 1) \ominus 1 = x$.
		\item \label[axiom]{ax:P3} $(x \ominus 1) + 1 = x \lor 1$.
		\item \label[axiom]{ax:P4} There exists $n \in \N$ such that $x \leq \underbrace{1 + \dots + 1}_{n \ \text{times}}$.
	\end{enumerate}
\end{definition}

We denote with $\ULMP$ the category of {\pulms} with homomorphisms.
Given $n \in \N$, we write $n$ for $\underbrace{1 + \dots + 1}_{n \ \text{times}}$.

In this section, we show that $\ULM$ and $\ULMP$ are equivalent.

\begin{lemma} \label{l:n is cancellative}
	Let $M$ be {\apulm}.
	For all $x, y \in M$ and every $n \in \N$, if $x + n = y + n$ then $x = y$.
\end{lemma}

\begin{proof}
	The proof proceeds by induction on $n \in \N$.
	The case $n = 0$ is trivial.
	Suppose the statement holds for $n \in \N$.
	If $x + (n + 1) = y + (n + 1)$, then 
	\[
		x + n \stackrel{\text{\cref{ax:P2}}}{ = } (x + n + 1) \ominus 1 = (y + n + 1) \ominus 1 \stackrel{\text{\cref{ax:P2}}}{ = } y + n,
	\]
	and then, by inductive hypothesis, $x = y$.
\end{proof}

\begin{remark} \label{r:impl def}
	Let $x$ and $y$ be elements of {\apulm}.
	Then
	\[
		y = x \ominus 1 \stackrel{\text{\cref{l:n is cancellative}}}{\Longleftrightarrow} y + 1 = (x \ominus 1) + 1 \stackrel{\text{\cref{ax:P3}}}{\Longleftrightarrow} y + 1 = x \lor 1.
	\]
	Moreover,
	\[
		y = 0 \stackrel{\text{\cref{l:n is cancellative}}}{\Longleftrightarrow} y + 1 = 1.
	\]
	This shows that the unary operation $-\ominus 1$ and the constant $0$ can be explicitly defined from $\{\lor, \land, 1, +\}$.
	Therefore, every function $f\colon M \to N$ between {\pulms} that preserves $+$, $\lor$, $\land$ and $1$ preserves also $-\ominus 1$ and $0$, and hence it is a homomorphism.
\end{remark}

Given {\aulm} $M$, we endow $M^+ $ with the operations $+$, $\lor$, $\land$, $0$, $1$ defined by restriction and with $- \ominus 1$ defined by $x \ominus 1 \df (x - 1) \lor 0$.
The restriction of $ + $ on $M^+ $ is well-defined by \cref{l:sum positive}; it is immediate that $\lor$, $\land$ and $- \ominus 1$ are well-defined, and that $0, 1 \in M^+ $.

\begin{proposition} \label{p:M^+ is positive-m}
	For every {\ulm} $M$, the algebra $M^+ $ is {\apulm}.
\end{proposition}

\begin{proof}
	The algebra $\langle M^+ ; +, \lor, \land, 0 \rangle$ is {\alm} because it is a subalgebra of the {\lm} $\langle M; +, \lor, \land, 0 \rangle$; so, \cref{ax:P0} holds.
	\Cref{ax:P1} holds by definition of $M^+$.
	\Cref{ax:P2} holds because, for every $x \in M^+ $, we have
	\[
		(x + 1) \ominus 1 = (x + 1 - 1) \lor 0 = x \lor 0 = x.
	\]
	\Cref{ax:P3} holds because, for every $x \in M^+ $, we have
	\[
		(x \ominus 1) + 1 = ((x - 1) \lor 0) + 1 = x \lor 1.
	\]
	\Cref{ax:P4} holds because, by \cref{ax:U3}, for every $x \in M^+ $ there exists $n \in \N$ such that $x \leq n$.
\end{proof}

Given a morphism $f\colon M \to N$ of {\ulms}, $f$ restricts to a function $f^+ $ from $M^+ $ to $N^+ $.
Moreover, $f$ preserves $+$, $\lor$, $\land$ and $1$ and so, by \cref{r:impl def}, $f^+$ is a morphism of {\pulms}.
This establishes a functor
$
	(-)^+ \colon \ULM \to \ULMP
$
that maps $M$ to $M^+ $, and maps a morphism $f \colon M \to N$ to its restriction $f^+ \colon M^+ \to N^+ $.
We will prove that $(-)^+ $ is an equivalence of categories (\cref{t:G_0 is equiv} below).
To do so, we exhibit a quasi-inverse $\T\colon \ULMP\to \ULM$.

Let $M$ be a {\pulm}.
We want to construct a {\ulm} $\T(M)$ such that, if $N$ is a {\ulm} and $N^+ \cong M$, then $\T(M) \cong N$.
Every element of a {\ulm} $N$ can be expressed as $x - n$ for some $x \in N^+ $ and $n \in \N$.
Roughly speaking, we will obtain $\T(N^+ ) \cong N$ by translating the elements of $N^+ $ by negative integers.
(In fact, $\T$ stands for `translations'.)

Therefore, given a {\ulm} $M$, we consider the relation $\sim$ defined on $M \times \N$ as follows: $(x,n) \sim(y,m)$ if, and only if, $x + m = y + n$.
Using \cref{l:n is cancellative}, it is not difficult to show that $\sim$ is an equivalence relation.
The equivalence class of an element $(x,n)$ of $M \times \N$ is denoted by $[(x, n)]$, or simply $[x, n]$.
We set $\T(M) \df \frac{M \times \N}{\sim}$, and we endow $\T(M)$ with the operations of {\aulm}:
\begin{align*}
	0							& \df [0, 0];\\
	1							& \df [1, 0];\\
	-1							& \df [0, 1];\\
	[x, n] + [y, m]		& \df [x + y, n + m];\\
	[x, n] \lor [y, m]	& \df [(x + m) \lor (y + n), n + m];\\
	[x, n] \land [y, m]	& \df [(x + m) \land (y + n), n + m].
\end{align*}
Straightforward computations show that these operations are well-defined.

\begin{remark} \label{r:we can add}
	For every element $x$ of {\apulm} and all $n,m \in \N$ we have $(x,n) \sim (x + m,n + m)$.
\end{remark}

\begin{lemma} \label{l:wlog n = m}
	Let $M$ be {\apulm}.
	For all $x, y \in M$ and every $n \in \N$ we have
	\begin{align*}
		[x, n] \lor [y, n]	& = [x \lor y, n],
	\intertext{and}
		[x,n] \land [y,n]		& = [x \land y, n].
	\end{align*}
\end{lemma}

\begin{proof}
	We have
	\begin{align*}
		[x, n] \lor [y, n]	& = [(x + n) \lor (y + n), n + n]	&&\\
									& = [(x \lor y) + n, n + n]			&&\by{$ + $ distr.\ over $ \lor $} \\
									& = [x \lor y, n], 						&&\by{\cref{r:we can add}}
	\end{align*}
	and analogously for $ \land $.
\end{proof}

\begin{proposition}
	For every {\pulm} $M$, $\T(M)$ is a {\ulm}.
\end{proposition}

\begin{proof}
	The fact that $\T(M)$ is a commutative monoid follows from the fact that $M$ and $\N$ are commutative monoids.
	Checking that $\T(M)$ is a distributive lattice is facilitated by \cref{r:we can add,l:wlog n = m}.
	Let us prove that $ + $ distributes over $ \lor $:
	\begin{align*}
		[x, n] + ([y, m] \lor [z, m])	& = [x, n] + [y \lor z, m]								&& \by{\cref{l:wlog n = m}} \\
												& = [x + (y \lor z), n + m]							&& \\
												& = [(x + y) \lor (x + z), n + m]					&& \\
												& = [x + y, n + m] \lor [x + z, n + m]				&& \by{\cref{l:wlog n = m}} \\
												& = ([x, n] + [y, m]) \lor ([x, n] + [z, m]).	&&
	\end{align*}
	Analogously for $ \land $.
	The axioms for $1$ and $-1$ are easily seen to hold.
\end{proof}

For a morphism of {\pulms} $f\colon M \to N$, we set
\begin{align*}
	\T(f) \colon	\T(M)		& \longrightarrow	\T(N) \\
						[x, n]	& \longmapsto 		[f(x), n].
\end{align*}
The function $\T(f)$ is well-defined: indeed, if $(x, n) \sim (y, m)$, then $x + m = y + n$, and then $f(x) + m = f(x + m) = f(y + n) = f(y) + n$, and therefore $(f(x), n) \sim (f(y), m)$.
Moreover, $\T(f)$ is a morphism of \ulms.
We show only that $ + $ is preserved:
\begin{align*}
	\T(f)([x, n] + [y, m])	& = \T(f)([x + y, n + m]) \\
										& = [f(x + y), n + m]\\
										& = [f(x) + f(y), n + m]\\
										& = [f(x), n] + [f(y), m].
\end{align*}
One easily verifies that $\T\colon \ULMP\to \ULM$ is a functor.

For each {\ulm}, we consider the function
\begin{align*}
	\eps_M^0 \colon	\T(M^+ )	& \longrightarrow	M \\
							[x,n]		& \longmapsto		x - n.
\end{align*}
The function $\eps_M^0$ is well-defined: indeed, if $[x, n] = [y,m]$, then $x + m = y + n$ and therefore $x - n = y - m$.

\begin{proposition} \label{p:eps0 is morphism}
	The function $\eps_M^0 \colon \T(M^+ ) \to M$ is a morphism of {\ulms} for every {\ulm}.
\end{proposition}

\begin{proof}
	The function $\eps_M^0$ preserves $0$, $1$, and $-1$ because $\eps_M^0([0, 0]) = 0 - 0 = 0$, $\eps_M^0([1, 0]) = 1 - 0 = 1$ and $\eps_M^0([0, 1]) = 0 - 1 = -1$.
	For all $x, y \in M^+$ and $n \in \N$ we have
	\begin{align*}
		\eps_M^0([x, n] + [y, m])	& = \eps_M^0([x + y, n + m]) \\
											& = (x + y) - (n + m) \\
											& = (x - n) + (y - m) \\
											& = \eps_M^0([x, n]) + \eps_M^0([y, m]).
	\end{align*}
	Hence, $\eps_M^0$ preserves $ + $.
	Moreover, for all $x, y \in M^+$ and $n, m \in \N$, we have
	\begin{align*}
	 	\eps_M^0([x, n] \lor [y, m])	& = \eps_M^0([(x + m) \lor (y + n), n + m]) \\
												& = ((x + m) \lor (y + n)) - (n + m) \\
												& = (x + n) \lor (y + m).							&& \by{$ + $ distr.\ over $ \lor $}
	\end{align*}
	Hence, $\eps_M^0$ preserves $ \lor $.
	Analogously, $\eps_M^0$ preserves $ \land $.
\end{proof}

\begin{proposition} \label{p:eps0 is natural}
	$\eps^0 \colon \T(-)^+ \dot{\to} \Id_{\ULM}$ is a natural transformation, i.e., for every morphism of {\ulms} $f\colon M \to N$, the following diagram commutes.
	\[
		\begin{tikzcd}
			\T(M^+ )	\arrow[swap]{d}{\T(f^+ )} \arrow{r}{\eps^0_M}	& M \arrow{d}{f} \\
			\T(N^+ )	\arrow[swap]{r}{\eps^0_N}							& N
		\end{tikzcd}
	\]
\end{proposition}

\begin{proof}
	For every $x \in M^+ $ and every $n \in \N$ we have
	\begin{align*}
		\eps_N^0(\T(f^+)([x, n]))	& = \eps_N^0([f^+ (x), n]) \\
												& = \eps_N^0([f(x), n]) \\
												& = f(x)-n \\
												& = f(x - n) \\
												& = f(\eps_M^0([x, n])). \qedhere
	\end{align*}
\end{proof}

\begin{proposition} \label{p:eps0 is iso}
	The function $\eps_M^0 \colon \T(M^+ ) \to M$ is bijective for every {\ulm} $M$.
\end{proposition}

\begin{proof}
	To prove injectivity, let $x, y \in M^+ $, let $n,m\in \N$, and suppose we have $\eps_M^0([x,n]) = \eps_M^0([y,m])$.
	Then, 
	\[
		x - n = \eps_M^0([x,n]) = \eps_M^0([y,m]) = y - m;
	\]
	hence $x + m = y + n$, and thus $[x,n] = [y,m]$;
	this proves injectivity.
	To prove surjectivity, for $x \in M$, choose $n \in \N$ such that $-n \leq x$; then $\eps_M^0([x + n,n]) = x + n-n = x$.
\end{proof}

For each {\pulm} $M$, we consider the function 
\begin{align*}
	\eta_M^0 \colon M&\longrightarrow (\T(M))^+ \\
	x&\longmapsto [x, 0].
\end{align*}

\begin{proposition} \label{p:eta preserves}
	For every {\pulm} $M$, the function $\eta^0_M \colon M \to (\T(M))^+ $ is a morphism of {\pulms}.
\end{proposition}

\begin{proof}
	It is easy to see that $\eta_M^0$ preserves $1$, $ + $, $ \lor $ and $ \land $.
	Then, by \cref{r:impl def}, the function $\eta^0_M$ is a morphism of {\pulms}.
\end{proof}

\begin{proposition} \label{p:eta is natural}
	$\eta^0 \colon \Id_{\ULMP} \dot{\to}(-)^+ \T$ is a natural transformation, i.e., for every morphism of {\pulms} $f\colon M \to N$, the following diagram commutes.
	\[
		\begin{tikzcd}
			M	\arrow[swap]{d}{f} \arrow{r}{\eta^0_M}	& (\T(M))^+ \arrow{d}{(\T(f))^+ } \\
			N	\arrow[swap]{r}{\eta^0_N}					& (\T(N))^+ 
		\end{tikzcd}
	\]
\end{proposition}

\begin{proof}
	For every $x \in M$, we have
	\[
	 (\T(f))^+(\eta_M^0(m)) = (\T(f))^+([x, 0]) = \T(f)([x, 0]) = [f(x), 0] = \eta_N^0(f(x)). \qedhere
	\]
\end{proof}

\begin{notation}
	Let $M$ be {\apulm}.
	We define, inductively on $n \in \N$, a function $(\;\cdot\;) \ominus n \colon M \to M$.
	\begin{align*}
												x \ominus 0	\df{}	& x;\\
		\text{(for $n \geq 1$)} \ \ \ 	x \ominus n	 = {}	& (x \ominus (n-1)) \ominus 1.
	\end{align*}
\end{notation}

\begin{lemma} \label{l:fixed}
	Let $M$ be {\apulm}.
	For every $x \in M$ and every $n \in \N$, we have
	\[
		(x \ominus n) + n = x \lor n.
	\]
\end{lemma}

\begin{proof}
	We prove the statement by induction.
	The case $n = 0$ is trivial.
	Suppose the statement holds for $n-1 \in \N$, and let us prove it for $n$.
	We have
	\begin{align*}
		(x \ominus n) + n	& = ((x \ominus (n-1)) \ominus 1) + n				&& \by{ind.\ def.\ of $- \ominus n$} \\
								& = ((x \ominus (n-1)) \ominus 1) + 1 + (n-1)	&& \\
								& = ((x \ominus (n-1)) \lor 1) + (n-1)				&& \by{\cref{ax:P3}} \\
								& = ((x \ominus (n-1)) + (n-1)) \lor n				&& \by{$ + $ distr.\ over $ \lor $} \\
								& = x \lor (n-1) \lor n								&& \by{ind.\ hyp.} \\
								& = x \lor n. 												&& \qedhere
	\end{align*}
\end{proof}

\begin{proposition} \label{p:eta0 is iso}
	For every {\pulm} $M$, the function $\eta_M^0 \colon M \to (\T(M))^+ $ is bijective.
\end{proposition}

\begin{proof}
	First, we prove that $\eta_M^0$ is injective.
	Let $x, y \in M$, and suppose $\eta_M^0(x) = [x, 0] = [y, 0] = \eta_M^0(y)$.
	Then, $x + 0 = y + 0$, and so $x = y$.
	Second, we prove that $\eta_M^0$ is surjective.
	Let $[x,n]\in (\T(M))^+$.
	Then,
	\begin{align*}
		[x, n]	& = [x, n] \lor [0, 0]		&& \by{$[x,n]\in (\T(M))^+$}\\
					& = [x \lor n, n]\\
					& = [(x \ominus n) + n,n]	&& \by{\cref{l:fixed}}\\
					& = [x \ominus n, 0]			&& \by{\cref{r:we can add}}\\
					& = \eta_M^0(x \ominus n).	&& \qedhere
	\end{align*}
\end{proof}

We recall that two functors $F \colon \cat{A} \to \cat{B}$ and $G \colon \cat{B} \to \cat{A}$ are called \emph{quasi-inverses} if the functors $GF \colon \cat{A} \to \cat{A}$ and $FG \colon \cat{B} \to \cat{B}$ are naturally isomorphic to the identity functors on $\cat{A}$ and $\cat{B}$ respectively.
Two categories $\cat{A}$ and $\cat{B}$ are equivalent if, and only if, there exist two quasi-inverses $F \colon \cat{A} \to \cat{B}$ and $G \colon \cat{B} \to \cat{A}$ \cite[Chapter~IV, Section~4]{MacLane1998}.

\begin{theorem} \label{t:G_0 is equiv}
	The functors $(-)^+ \colon \ULM \to \ULMP$ and $\T\colon \ULMP \to \ULM$ are quasi-inverses.
\end{theorem}

\begin{proof}
	The functors $\Id_{\ULMP} \colon \ULMP\to \ULMP$ and $(-)^+\T \colon \ULMP\to \ULMP$ are naturally isomorphic by \cref{p:eta preserves,p:eta is natural,p:eta0 is iso}.
	The functors $\Id_{\ULM} \colon \ULM \to \ULM$ and $\T(-)^+ \colon \ULM \to \ULM$ are naturally isomorphic by \cref{p:eps0 is morphism,p:eps0 is natural,p:eps0 is iso}.
\end{proof}

%%%%%%%%%%%%%%%%%%%%%%%%%%%%%%%%%%% SECTION %%%%%%%%%%%%%%%%%%%%%%%%%%%%%%%%%%%%

\section{Good sequences} \label{s:good.def}

\begin{definition}
	Let $A$ be an {\mvm}.
	A \emph{good pair} in $A$ is a pair $(x_0, x_1)$ of elements of $A$ such that $x_0 \oplus x_1 = x_0$ and $x_0 \odot x_1 = x_1$.
	A \emph{good sequence} in $A$ is a sequence $(x_0, x_1, x_2, \dots)$ of elements of $A$ which is eventually $0$ and such that, for each $n \in \N$, $(x_n, x_{n + 1})$ is a good pair.
\end{definition}

Instead of $(x_0, \dots, x_n, 0, 0, 0, \dots)$ we shall often write, concisely, $(x_0, \dots, x_n)$.
Thus, if $0^m$ denotes an $m$-tuple of zeros, the good sequences $(x_1, \dots, x_n)$ and $(x_0, \dots, x_n, 0^m)$ are identical.
For each $x \in A$, the sequence $(x, 0, 0, 0, \dots)$ (which is always good) will be denoted by $(x)$.

\begin{remark}
	In our definition of good pair we included both the condition $x_0 \oplus x_1 = x_0$ and the condition $x_0 \odot x_1 = x_1$ because, in general, they are not equivalent.
	As an example, one can take the {\mvm} consisting of three elements $\{0, a, 1\}$, where $a \oplus a = a$, and $a \odot a = 0$.
\end{remark}

In order to prove the equivalence between the categories of {\mvs} and {\ulgs} (see \cite{Mundici} or \cite{Cignoli}), Mundici used the facts that subdirectly irreducible {\mvs} are totally ordered and that good sequences in totally ordered {\mvs} are of the form $(1, \dots, 1, x, 0, 0, \dots)$.

In this paper we do not make use of the Subdirect Representation Theorem (in fact, we do not make use of the axiom of choice) to establish the equivalence between $\ULM$ and $\MVM$.
The reason why this is done is that, initially, the author was unable to prove that, in subdirectly irreducible {\mvms}, good sequences are of the form $(1, \dots, 1, x, 0, 0, \dots)$.
Eventually such a proof was found, and the result is given in \cref{c:good sequence in sub irr}.
However, the result is not used in the present paper, for the following reasons.
First, in this way, the proof that we provide for the equivalence between $\ULM$ and $\MVM$ may be applied in similar settings, where the structure of subdirectly irreducible algebras is not known.
Secondly, the proof we give does not rely on the axiom of choice. 
In particular, \emph{up to proving without the axiom of choice} that the axioms of {\mvms} hold in any {\mv}, we obtain a proof of the equivalence between {\ulgs} and {\mvs} that does not make use of the axiom of choice.

For a proof of our main result that does take advantage of the Subdirect Representation Theorem, the reader is invited to consult the author's Ph.D.\ thesis \cite[Chapter~4]{Abbadini2021}.

%%%%%%%%%%%%%%%%%%%%%%%%%%%%%%%%%%% SECTION %%%%%%%%%%%%%%%%%%%%%%%%%%%%%%%%%%%%

\section{Basic properties of MV-monoidal algebras}

\begin{remark}
	Inspection of the axioms that define {\mvms} shows that, for every {\mvm} $\langle A; \oplus, \odot, \lor, \land, 0, 1 \rangle$, also its `dual' algebra $\langle A; \odot, \oplus, \land, \lor, 1, 0 \rangle$ is {\amvm}.
\end{remark}

We give a name to the right- and left-hand terms of \cref{ax:A4,ax:A5}; we will then prove that their interpretations in {\amvm} coincide.

\begin{notation}
	We set
	\begin{align*}
		\sigma_1(x, y, z) \df (x \oplus y) \odot ((x \odot y) \oplus z);\\
		\sigma_2(x, y, z) \df (x \odot y) \oplus ((x \oplus y) \odot z);\\
		\sigma_3(x, y, z) \df (x \odot (y \oplus z)) \oplus (y \odot z);\\
		\sigma_4(x, y, z) \df (x \oplus (y \odot z)) \odot (y \oplus z).
	\end{align*}
\end{notation}

Note that \cref{ax:A5} can be written as $\sigma_1(x, y, z) = \sigma_3(x, y, z)$, \cref{ax:A6} can be written as $\sigma_2(x, y, z) = \sigma_4(x, y, z)$, \cref{ax:A6} can be written as $(x \odot y) \oplus z = \sigma_1(x, y, z) \lor z$, and \cref{ax:A7} can be written as $(x \oplus y) \odot z = \sigma_2(x, y, z) \land z$.

\begin{lemma} \label{l:permutations}
	Let $A$ be {\amvm}. For all $i,j \in \{1,2,3,4\}$, every permutation $\rho \colon \{1,2,3\} \to \{1,2,3\}$ and all $x_1, x_2, x_3 \in A$ we have
	\[
		\sigma_i(x_{1}, x_{2}, x_{3}) = \sigma_j(x_{\rho(1)}, x_{\rho(2)}, x_{\rho(3)}).
	\]
	In other words, the terms $\sigma_1$, $\sigma_2$, $\sigma_3$, $\sigma_4$ in the theory of {\mvms} are all invariant under permutations of variables, and they coincide.
\end{lemma}
\begin{proof}
	By commutativity of $\oplus$ and $\odot$, in the theory of {\mvms} $\sigma_1$ is invariant under transposition of the first and the second variables, and $\sigma_3$ is invariant under transposition of the second and the third ones.
	Moreover, by \cref{ax:A4}, we have $\sigma_1(x, y, z) = \sigma_3(x, y, z)$.
	Since any two distinct transpositions in the symmetric group on three elements generate the whole group, it follows that $\sigma_1$ and $\sigma_3$ are invariant under any permutation of the variables.
	By commutativity of $\oplus$ and $\odot$, we have $\sigma_1(x, y, z) = \sigma_4(z, y, x)$, and, by \cref{ax:A5}, we have $\sigma_2(x, y, z) = \sigma_4(x, y, z)$.
	We conclude that $\sigma_1$, $\sigma_2$, $\sigma_3$, $\sigma_4$ are invariant under permutations of variables, and they coincide.
\end{proof}

In particular, \cref{l:permutations} guarantees that we have
\[
	\sigma_1(x, y, z) = \sigma_2(x, y, z) = \sigma_3(x, y, z) = \sigma_4(x, y, z).
\]

\begin{notation}
	For $x, y, z$ in {\amvm}, we let $\sigma(x, y, z)$ denote the common value of $\sigma_1(x, y, z)$, $\sigma_2(x, y, z)$, $\sigma_3(x, y, z)$ and $\sigma_4(x, y, z)$.
\end{notation}

\begin{lemma} \label{l:bounded-lattice}
	For every $x$ in {\amvm} we have $0 \leq x \leq 1$.
\end{lemma}
\begin{proof}
	We have
	\begin{align*}
		x	
			& = (x \odot 1) \oplus 0
			&& \by{\cref{ax:A2}} \\
			& = \sigma_1(x, 1, 0) \lor 0
			&& \by{\cref{ax:A6}} \\
			& = \sigma_3(x, 1, 0) \lor 0
			&& \by{\cref{ax:A4}} \\
			& = ((x \odot (1 \oplus 0)) \oplus (1 \odot 0)) \lor 0
			&& \by{def.\ of $\sigma_3$} \\
			& = ((x \odot 1) \oplus 0) \lor 0
			&& \by{\cref{ax:A2}} \\
			& = x \lor 0.
			&& \by{\cref{ax:A2}}
	\end{align*}
	Thus, $0 \leq x$.
	Dually, $x \leq 1$.
\end{proof}

\begin{lemma} \label{l:absorbing}
	For every $x$ in {\amvm}, we have $x \oplus 1 = 1$ and $x \odot 0 = 0$.
\end{lemma}
\begin{proof}
	We have
	\begin{align*}
		0	& = 1 \odot 0								&& \by{\cref{ax:A2}} \\
			& = (1 \lor x) \odot 0					&& \by{\cref{l:bounded-lattice}} \\
			& = (1 \odot 0) \lor (x \odot 0)		&& \by{\cref{ax:A3}} \\
			& = 0 \lor (x \odot 0)					&& \by{\cref{ax:A2}} \\
			& = x \odot 0.								&& \by{\cref{l:bounded-lattice}}
	\end{align*}
	Dually, $x \oplus 1 = 1$.
\end{proof}

\begin{lemma}
	\label{l:order-preserving properties}
	The following properties hold for all $x, y, z, x', y'$ in {\amvm}.
	\begin{enumerate}
		\item
			\label{i:basic1} If $x \leq x'$ and $y \leq y'$, then $x \oplus y \leq x' \oplus y'$.
		\item
			\label{i:basic2} If $x \leq x'$ and $y \leq y'$, then $x \odot y \leq x' \odot y'$.
		\item
			\label{i:basic3} $x \leq y\Rightarrow x \oplus z \leq y \oplus z$
		\item
			\label{i:basic4} $x \leq y\Rightarrow x \odot z \leq y \odot z$.
		\item
			\label{i:basic5} $x \leq x \oplus y$.
		\item
			\label{i:basic6} $x \geq x \odot y$.

	\end{enumerate}
\end{lemma}
\begin{proof}
	Let us call $A$ the {\mvm} of the statement.
	\Cref{i:basic1} is guaranteed by the application of \cref{l:sum positive} to the {\lm} $\langle A; \oplus, \lor, \land, 0 \rangle$.
	\Cref{i:basic2} is dual to \cref{i:basic1}.
	\Cref{i:basic3} holds by \cref{i:basic1} together with the fact that $z \leq z$.
	\Cref{i:basic4} is dual to \cref{i:basic3}.
	From $0 \leq y$ we obtain, by \cref{i:basic3}, $x \oplus 0 \leq x \oplus y$.
	By \cref{l:absorbing}, we have $x \oplus 0 = x$.
	Therefore, $x \leq x \oplus y$, and so \cref{i:basic5} is proved.
	\Cref{i:basic6} is dual to \cref{i:basic5}.
\end{proof}

\begin{lemma} \label{l:almost associative}
	For all $x$, $y$, $z$ in {\amvm} we have
	\[
		x \odot (y \oplus z) \leq (x \odot y) \oplus z.
	\]
\end{lemma}

\begin{proof}
	Using \cref{ax:A4,ax:A5,ax:A6,ax:A7}, we obtain
	\[
		x \odot (y \oplus z) = x \land \sigma(x, y, z) \leq \sigma(x, y, z) \leq \sigma(x, y, z) \lor z = (x \odot y) \oplus z. \qedhere
	\]
\end{proof}

\begin{lemma} \label{l:switch of plus and dot if good}
	Let $A$ be an {\mvm}, let $(x_0, x_1)$ be a good pair in $A$, and let $y \in A$.
	Then
	\[
		x_0 \odot (x_1 \oplus y) = x_1 \oplus (x_0 \odot y),
	\]
	and both these elements coincide with $\sigma(x_0, x_1, y)$.
\end{lemma}

\begin{proof}
	We have
	\begin{align*}
		x_0 \odot (x_1 \oplus y)
			& = (x_0 \oplus x_1) \odot ((x_0 \odot x_1) \oplus y)
			&& \by{$(x_0, x_1)$ is good} \\
			& = \sigma_1(x_0, x_1, y)
			&& \by{def.\ of $\sigma_1$} \\
			& = \sigma_2(x_0, x_1, y)
			&& \by{\cref{l:permutations}} \\
			& = (x_0 \odot x_1) \oplus ((x_0 \oplus x_1) \odot y)
			&& \by{def.\ of $\sigma_2$} \\
			& = x_1 \oplus (x_0 \odot y).
			&& \by{$(x_0, x_1)$ is good} \qedhere
	\end{align*}
\end{proof}

\begin{lemma} \label{l:oplus odot is a good pair}
	For all $x$ and $y$ in {\amvm}, the pair $(x \oplus y, x \odot y)$ is good.
\end{lemma}

\begin{proof}
	We have
	\begin{align*}
		(x \oplus y) \oplus (x \odot y)	& = (1 \odot (x \oplus y)) \oplus (x \odot y)	&& \by{\cref{ax:A2}} \\
													& = \sigma_3(1, x, y)										&& \by{def.\ of $\sigma_3$} \\
													& = \sigma_4(1, x, y)										&& \by{\cref{l:permutations}} \\
													& = (1 \oplus (x \odot y)) \odot (x \oplus y)	&& \by{def.\ of $\sigma_4$} \\
													& = 1 \odot (x \oplus y)								&& \by{\cref{l:absorbing}} \\
													& = x \oplus y.											&& \by{\cref{ax:A2}}
	\end{align*}
	Dually,
	$(x \oplus y) \odot (x \odot y) = x \odot y$.
\end{proof}

%%%%%%%%%%%%%%%%%%%%%%%%%%%%%%%%%%% SECTION %%%%%%%%%%%%%%%%%%%%%%%%%%%%%%%%%%%%

\section{Operations on the set of good sequences} \label{s:operations on Good}

We denote with $\GS(A)$ the set of good sequences in {\amvm} $A$. (In fact, $\GS$ stands for `good'.)
We will endow $\GS(A)$ with the structure of {\apulm}.
We let $\gs{0}$ denote the good sequence $(0, 0, 0, \dots)$, and we let $\gs{1}$ denote the good sequence $(1, 0, 0, 0, \dots)$.
For good sequences $\gs{a} = (a_0, a_1, a_2, \dots)$ and $\gs{b} = (b_0, b_1, b_2, \dots)$, we set
\[
	\gs{a} \lor \gs{b} \df (a_0 \lor b_0, a_1 \lor b_1, a_2 \lor b_2, \dots),
\]
and
\[
	\gs{a} \land \gs{b} \df (a_0 \land b_0, a_1 \land b_1, a_2 \land b_2, \dots).
\]
\Cref{p:join is good} below asserts that $\gs{a} \lor \gs{b}$ and $\gs{a} \land \gs{b}$ are good sequences.
In order to prove it, we establish the following lemmas.

\begin{lemma} \label{l:join of good pairs}
	For all good pairs $(x_0, x_1)$ and $(y_0, y_1)$ in {\amvm}, the pairs $(x_0 \lor y_0, x_1 \lor y_1)$ and $(x_0 \land y_0, x_1 \land y_1)$ are good.
\end{lemma}

\begin{proof}
	We prove that $(x_0 \lor y_0, x_1 \lor y_1)$ is a good pair.	We have
	\begin{align*}
		& (x_0 \lor y_0) \oplus (x_1 \lor y_1) \\
		& = (x_0 \oplus x_1) \lor (x_0 \oplus y_1) \lor (y_0 \oplus x_1) \lor (y_0 \oplus y_1)		&& \by{$\oplus$ distr.\ over $ \lor $} \\
		& = x_0 \lor (x_0 \oplus y_1) \lor (y_0 \oplus x_1) \lor y_0										&& \by{$(x_0, x_1)$, $(y_0, y_1)$ good} \\
		& = (x_0 \oplus y_1) \lor (y_0 \oplus x_1)																&& \by{\cref{l:order-preserving properties}.\eqref{i:basic5}} \\
		& = (x_0 \oplus (y_0 \odot y_1)) \lor (y_0 \oplus (x_0 \odot x_1))								&& \by{$(x_0, x_1)$, $(y_0, y_1)$ good} \\
		& = x_0 \lor \sigma(x_0, y_0, y_1) \lor y_0 \lor \sigma(x_0, x_1, y_0) 							&& \by{\cref{ax:A6}} \\
		& = x_0 \lor ((x_0 \oplus y_1) \odot y_0) \lor y_0 \lor (x_0 \odot (x_1 \oplus y_0))		&& \by{\cref{l:switch of plus and dot if good}}\\
		& = (x_0 \lor (x_0 \odot (x_1 \oplus y_0))) \lor (((x_0 \oplus y_1) \odot y_0) \lor y_0)	&&	 \\
		& = x_0 \lor y_0.																									&& \by{\cref{l:order-preserving properties}.\eqref{i:basic6}}
	\end{align*}
	Moreover, we have
	\begin{align*}
		&(x_0 \lor y_0) \odot (x_1 \lor y_1) \\
		& = (x_0 \odot x_1) \lor (x_0 \odot y_1) \lor (y_0 \odot x_1) \lor (y_0 \odot y_1)	&& \by{$\oplus$ distr.\ over $ \lor $} \\
		& = x_1 \lor (x_0 \odot y_1) \lor (y_0 \odot x_1) \lor y_1									&& \by{$(x_0, x_1)$, $(y_0, y_1)$ good} \\
		& = x_1 \lor y_1.																							&& \by{\cref{l:order-preserving properties}.\eqref{i:basic6}}
	\end{align*}
	Hence, $(x_0 \lor y_0, x_1 \lor y_1)$ is good.
	Dually, $(x_0 \land y_0, x_1 \land y_1)$ is good.
\end{proof}

\begin{proposition} \label{p:join is good}
	For all good sequences $\gs{a}$ and $\gs{b}$ in {\amvm}, the sequences $\gs{a} \lor \gs{b}$ and $\gs{a} \land \gs{b}$ are good.
\end{proposition}
\begin{proof}
	By \cref{l:join of good pairs}.
\end{proof}

\begin{proposition} \label{p:distributive lattice}
	Let $A$ be an {\mvm}.
	Then, $\langle\GS(A); \lor, \land \rangle$ is a distributive lattice.
\end{proposition}

\begin{proof}
	$\langle \GS(A); \lor, \land \rangle$ is a distributive lattice, because $ \lor $ and $ \land $ are applied componentwise, and $\langle A; \lor, \land \rangle$ is a distributive lattice.
\end{proof}

\noindent For $A$ an {\mvm}, we have a partial order $ \leq $ on $\GS(A)$, induced by the lattice operations.
Since the lattice operations are defined componentwise, we have the following.

\begin{remark} \label{r:order is pointwise}
	Let $A$ be {\amvm}. For all good sequences $\gs{a} = (a_0, a_1, a_2, \dots)$ and $\gs{b} = (b_0, b_1, b_2, \dots)$ in $A$, we have $\gs{a} \leq \gs{b}$ if, and only if, for all $n \in \N$, $a_n \leq b_n$.
\end{remark}

Now we want to define the sum of good sequences.
Given two good sequences $\gs{a} = (a_0, a_1, a_2, \dots)$ and $\gs{b} = (b_0, b_1, b_2, \dots)$ in {\amvm}, there are two natural ways to define a sequence $\gs{c} = (c_0, c_1, c_2, \dots)$ as the sum of $\gs{a}$ and $\gs{b}$.
The first one is
\[
	c_n \df (a_0 \oplus b_n) \odot (a_1 \oplus b_{n-1}) \odot \dots \odot (a_{n-1} \oplus b_1) \odot (a_n \oplus b_0),
\]
and the second one is
\[
	c_n \df b_n \oplus (a_0 \odot b_{n-1}) \oplus (a_1 \odot b_{n-2}) \oplus \dots \oplus (a_{n-2} \odot b_1) \oplus (a_{n-1} \odot b_0) \oplus a_n.
\]
Our first aim, reached in \cref{l:associative} below, is to show that these two ways coincide.

\begin{lemma} \label{l:bipartite}
%	Let $n,m\in \N$, 
	Let $x_0, \dots, x_n, y_0, \dots, y_m$ be elements of {\amvm} and suppose that, for every $i \in \{0, \dots,n\} $ and every $j \in \{0, \dots, m\}$, the pair $(x_i, y_j)$ is good.
	Then, the pair
	\[
		(x_0 \odot \dots \odot x_n, y_0 \oplus \dots \oplus y_m)	
	\]
	is good.
\end{lemma}

\begin{proof}
	The statement is trivial for $(n,m) = (0, 0)$.
	The statement is true for $(n,m) = (1, 0)$ because
	\[
		(x_0 \odot x_1) \oplus y_0 \stackrel{\text{\cref{l:switch of plus and dot if good}}}{ = } x_0 \odot (x_1 \oplus y_0) = x_0 \odot x_1,
	\]
	and
	\[
		(x_0 \odot x_1) \odot y_0 = x_0 \odot (x_1 \odot y_0) = x_0 \odot y_0 = y_0.
	\]
	The case $(n,m) = (0, 1)$ is analogous.
	Let $(n,m) \in \N \times \N\setminus \{(0, 0), (0, 1), (1, 0)\}$, and suppose that the statement is true for each $(h,k) \in \N \times \N$ such that $(h,k) \neq (n,m)$, $h \leq n$ and $k \leq m$.
	We prove that the statement holds for $(n,m)$.
	At least one of the two conditions $n \neq 0$ and $m\neq 0$ holds.
	Suppose, for example, $n \neq 0$.
	Then, by inductive hypothesis, the pairs $(x_0 \odot \dots \odot x_{n-1}, y_0 \oplus \dots \oplus y_m)$ and $(x_{n}, y_0 \oplus \dots \oplus y_m)$ are good.
	Now we apply the statement for the case $(1, 0)$, and we obtain that $(x_0 \odot \dots \odot x_n, y_0 \oplus \dots \oplus y_m)$ is a good pair.
	The case $m\neq 0$ is analogous.
\end{proof}

\begin{lemma} \label{l:oplus and odot are still good}
	Let $A$ be {\amvm}. For every good pair $(x, y)$ in $A$ and every $u \in A$, the pairs $(x \oplus u, y)$ and $(x, y \odot u)$ are good.
\end{lemma}

\begin{proof}
	The pair $(x \oplus u, y)$ is good because we have $(x \oplus u) \oplus y = (x \oplus y) \oplus u = x \oplus u$, and, by \cref{l:order-preserving properties}, we have $y = x \odot y \leq (x \oplus u) \odot y \leq y$,
	which implies $(x \oplus u) \odot y = y$.
	Dually, $(x, y \odot u)$ is good.
\end{proof}

\begin{lemma} \label{l:good is transitive}
	If $(x, y)$ and $(y, z)$ are good pairs in {\amvm}, then $(x, z)$ is a good pair.
\end{lemma}

\begin{proof}
	We have
	\[
		x \oplus z = (x \oplus y) \oplus z = x \oplus (y \oplus z) = x \oplus y = x,
	\]
	and
	\[
		x \odot z = x \odot (y \odot z) = (x \odot y) \odot z = y \odot z = z. \qedhere
	\]
\end{proof}

\begin{lemma} \label{l:associative}
	Let $A$ be {\amvm}, let $m \in \N$ and let $(a_0, \dots, a_m)$ and $(b_0, \dots, b_m)$ be good sequences in $A$.
	Then
	\[
		(a_0 \odot b_m) \oplus \dots \oplus (a_m \odot b_0) = a_0 \odot (a_1 \oplus b_m) \odot \dots \odot (a_m \oplus b_1) \odot b_0,
	\]
	and
	\[
		(a_0 \oplus b_m) \odot \dots \odot (a_m \oplus b_0) = b_m \oplus (a_0 \odot b_{m-1}) \oplus \dots \oplus (a_{m-1} \odot b_0) \oplus a_m.
	\]
\end{lemma}

\begin{proof}
	We prove the first equality by induction on $m\in 	\N$.
	The case $m = 0$ is trivial.
	Let $m\in \N\setminus \{0\}$ and suppose that the statement holds for $n-1$.
	By the inductive hypothesis, we have
	\begin{equation} \label{e:1}
		(a_0 \odot b_m) \oplus \dots \oplus (a_m \odot b_0) = (a_0 \odot (a_1 \oplus b_m) \odot \dots \odot (a_{m-1} \oplus b_2) \odot b_1) \oplus (a_m \odot b_0).
	\end{equation}
	By \cref{l:oplus and odot are still good}, the pair $(b_0, a_0 \odot (a_1 \oplus b_m) \odot \dots \odot (a_{m-1} \oplus b_2) \odot b_1)$ is good.
	Therefore, by \cref{l:switch of plus and dot if good}, we have
	\begin{multline} \label{e:2}
		(a_0 \odot (a_1 \oplus b_m) \odot \dots \odot (a_{m-1} \oplus b_2) \odot b_1) \oplus (a_m \odot b_0) \\
		 = ((a_0 \odot (a_1 \oplus b_m) \odot \dots \odot (a_{m-1} \oplus b_2) \odot b_1) \oplus a_m) \odot b_0.
	\end{multline}
	By \cref{l:good is transitive}, the pairs $(a_0, a_m)$, $(a_1, a_m)$, \dots $(a_{m-2}, a_m)$, $(a_{m-1}, a_m)$ are good.
	By \cref{l:oplus and odot are still good}, the pairs $(a_0, a_m)$, $(a_1 \oplus b_m, a_m)$, \dots, $(a_{m-2} \oplus b_3, a_m)$, $(a_{m-1} \oplus b_2, a_m)$ are good.
	By \cref{l:bipartite}, the pair 
	\[
		(a_0 \odot (a_1 \oplus b_m) \odot \dots \odot (a_{m-1} \oplus b_2), a_m)
	\]
	is good.
	Thus, by \cref{l:switch of plus and dot if good}, we have
	\begin{multline} \label{e:3}
		((a_0 \odot (a_1 \oplus b_m) \odot \dots \odot (a_{m-1} \oplus b_2) \odot b_1) \oplus a_m) \odot b_0 \\
		 = ((a_0 \odot (a_1 \oplus b_m) \odot \dots \odot (a_{m-1} \oplus b_2)) \odot (b_1 \oplus a_m)) \odot b_0.
	\end{multline}
	The chain of equalities established in \cref{e:1,e:2,e:3} settles the first equality of the statement.
	The second one is dual.
\end{proof}

Given two good sequences $\gs{a} = (a_0, a_1, a_2, \dots)$ and $\gs{b} = (b_0, b_1, b_2, \dots)$ in {\amvm}, we set
\[
	\gs{a} + \gs{b} \df(c_0, c_1, c_2, \dots),
\]
where
\[
	c_n \df (a_0 \oplus b_n) \odot (a_1 \oplus b_{n-1}) \odot \dots \odot (a_{n-1} \oplus b_1) \odot (a_n \oplus b_0),
\]
or, equivalently (by \cref{l:associative}),
\[
	c_n \df b_n \oplus (a_0 \odot b_{n-1}) \oplus (a_1 \odot b_{n-2}) \oplus \dots \oplus (a_{n-2} \odot b_1) \oplus (a_{n-1} \odot b_0) \oplus a_n.
\]

In \cref{p:sum is good} below, we show that $\gs{a} + \gs{b}$ is a good sequence.
In preparation for it, we establish the following lemma.

\begin{lemma} \label{l:hsr}
	Let $A$ be {\amvm} and let $(a_0, \dots, a_n)$ and $(b_0, \dots, b_n)$ be good sequences in $A$.
	Then, the pair
	\[
		((a_0 \oplus b_n) \odot \dots \odot (a_n \oplus b_0), (a_0 \odot b_n) \oplus \dots \oplus (a_n \odot b_0))
	\]
	is good.
\end{lemma}

\begin{proof}
	By \cref{l:bipartite}, it is enough to show that, for $i,j\in \{0, \dots, m\}$, the pair $(a_i \oplus b_{m-i}, a_j \odot b_{m-j})$ is good.
	The case $i = j$ is covered by \cref{l:oplus odot is a good pair}.
	If $i<j$, then, by \cref{l:good is transitive}, the pair $(a_i, a_j)$ is good; by \cref{l:oplus and odot are still good}, the pair $(a_i \oplus b_{m-i}, a_j \odot b_{m-j})$ is good.
	If $i>j$, then, by \cref{l:good is transitive}, the pair $(b_{m-i}, b_{m-j})$ is good; by \cref{l:oplus and odot are still good}, the pair $(a_i \oplus b_{m-i}, a_j \odot b_{m-j})$ is good.
\end{proof}

\begin{proposition} \label{p:sum is good}
	For all good sequences $\gs{a}$ and $\gs{b}$ in {\amvm}, the sequence $\gs{a} + \gs{b}$ is good.
\end{proposition}

\begin{proof}
	Let $\gs{a} = (a_0, \dots, a_m)$ and $\gs{b} = (b_0, \dots, b_m)$.
	Set $\gs{c} \df \gs{a} + \gs{b}$, and write $\gs{c} = (c_0, c_1, c_2, \dots)$.
	Then, for $j > 2m$ we have $c_j = 0$.
	For all $n \in \N$ we have
	\begin{align*}
		c_n& = (a_0 \oplus b_n) \odot \dots \odot (a_n \oplus b_0) \\
		& = (1 \oplus b_{n + 1}) \odot (a_0 \oplus b_n) \odot \dots \odot (a_n \oplus b_0) \odot (a_{n + 1} \oplus 1),
	\end{align*}
	and
	\[
		c_{n + 1} = (1 \odot b_{n + 1}) \oplus (a_0 \odot b_n) \oplus \dots \oplus (a_{n} \odot b_0) \oplus (a_{n + 1} \odot 1).
	\]
	By \cref{l:hsr}, the pair $(c_n, c_{n + 1})$ is good.
\end{proof}

\begin{proposition} \label{p:sum is commutative}
	Addition of good sequences is commutative.
\end{proposition}

\begin{proof}
	By commutativity of $\oplus$ and $\odot$, we have
	\[
		\gs{a} + \gs{b} = (a_0 \oplus b_n) \odot \dots \odot (a_n \oplus b_0) = (b_0 \oplus a_n) \odot \dots \odot (b_n \oplus a_0) = \gs{b} + \gs{a}. \qedhere
	\]
\end{proof}

\begin{remark} \label{r:0 is neutral}
	Let $A$ be an {\mvm}, and let $\gs{a} \in \GS(A)$.
	Then, $\gs{a} + \gs{0} = \gs{a}$.
\end{remark}

Now, we show that, for all good sequences $\gs{x}$, $\gs{y}$, $\gs{z}$ in {\amvm}, we have $(\gs{x} + \gs{y}) + \gs{z} = \gs{x} + (\gs{y} + \gs{z})$.
A direct verification, which seems difficult in general, becomes treatable when $\gs{y}$ is of the form $(y_0, 0, 0, \dots)$.
In fact, Light's associativity test guarantees that this is enough to imply associativity, thanks to the fact that the elements of the form $(y_0, 0, 0, \dots)$ generate $\GS(A)$.
In the following, we carry out the details.

\begin{lemma} \label{l:as sum of elements}
	For every good sequence $(a_0, \dots, a_n)$ in {\amvm} we have
	\[
		(a_0, \dots, a_n) = (a_0, \dots, a_{n-1}) + (a_n).
	\]
\end{lemma}

\begin{proof}
	Set $(b_0, b_1, b_2, \dots) \df (a_0, \dots, a_{n-1}) + (a_n)$.
	For $k\in \{0, \dots,n-1\}$, we have
	\[
		b_k = a_0 \odot a_1 \odot \dots \odot a_{k-1} \odot (a_k \oplus a_n) = a_0 \odot a_1 \odot \dots \odot a_{k-1} \odot a_k = a_k.
	\]
	Moreover, $b_n = a_0 \odot a_1 \odot \dots \odot a_{n-1} \odot a_n = a_n$, and, for $k>n$, we have $b_k = 0$. In conclusion, $(a_0, \dots, a_n) = (b_0, b_1, b_2, \dots) = (a_0, \dots, a_{n-1}) + (a_n)$.
\end{proof}

\begin{notation} \label{n:magma}
	A \emph{magma} $\langle X; \cdot\rangle$ consists of a set $X$ and a binary operation $\cdot$ on $X$.
	Given a subset $T$ of a magma $X$, we define, inductively on $n \in \Np$, the subset $T_n$; we set $T_1 \df T$, and, for $n \in \Np$, $T_n \df \{\,tz \mid t\in T, z\in T_{n-1}\,\} \cup\{\,zt \mid t\in T, z\in T_{n-1}\,\}$.
	Roughly speaking, $T_n$ is the set of elements of $X$ which can be obtained with at most $n$ occurrences of elements of $T$ via application of the operation $\cdot$.
	We say that $T$ \emph{generates} $X$ if $\bigcup_{n \in \Np}T_n = X$.
\end{notation}

\begin{lemma} \label{l:generation}
	For every {\mvm} $A$, the set $\{\,(x) \in \GS(A) \mid x \in A\,\}$ generates the magma $\langle \GS(A); + \rangle$.
\end{lemma}
\begin{proof}
	By induction, using \cref{l:as sum of elements}.
\end{proof}

\begin{lemma}[Light's associativity test]\label{l:Light}
	Let $\langle X; \cdot\rangle$ be a magma, and let $T$ be a subset of $X$ that generates $X$.
	Suppose that, for every $x, z\in X$ and $t\in T$, we have $(xt)z = x(tz)$.
	Then, the operation $\cdot$ is associative.
\end{lemma}

\begin{proof}
	Since $T$ generates $X$, we have $\bigcup_{n \in \Np}T_n = X$.
	We prove, by induction on $n \in \Np$, that, for every $y \in T_n$, and every $x, y \in X$, we have $(xy)z = x(yz)$.
	The case $n = 1$ is ensured by hypothesis.
	Let $n \geq 2$, and suppose that the cases $1, \dots,n-1$ hold.
	Then, either $y = ty'$ or $y = y't$, for some $t\in T$ and $y'\in T_n$.
	Suppose, for example, $y = ty'$.
	Then, we have
	\[
		(xy)z = (x(ty'))z = ((xt)y')z = (xt)(y'z) = x(t(y'z)) = x((ty')z) = x(yz).
	\]
	The case $y = y't$ is analogous.
\end{proof}

\begin{lemma} \label{l:associative 212}
	Let $A$ be {\amvm}, let $(a_0, a_1)$ and $(b_0, b_1)$ be good pairs in $A$, and let $x \in A$.
	Then
	\[
		(a_0 \oplus (x \odot b_0) \oplus b_1) \odot (a_1 \oplus x \oplus b_0) = (a_0 \oplus x \oplus b_1) \odot (a_1 \oplus (a_0 \odot x) \oplus b_0),
	\]
	and both sides coincide with $a_1 \oplus \sigma(a_0, x, b_0) \oplus b_1$.
\end{lemma}
\begin{proof}
	Since the pair $(a_0, a_1)$ is good, it follows from \cref{l:oplus and odot are still good} that the pair $(a_0 \oplus (x \odot b_0) \oplus b_1, a_1)$ is good.
	Since the pair $(b_0, b_1)$ is good, it follows from \cref{l:oplus and odot are still good} that the pair $(x \oplus b_0, b_1)$ is good.
	Therefore, we have
	\begin{align*}
		& (a_0 \oplus (x \odot b_0) \oplus b_1) \odot (a_1 \oplus x \oplus b_0) \\
		& = a_1 \oplus ((a_0 \oplus (x \odot b_0) \oplus b_1) \odot (x \oplus b_0))	&& \by{\cref{l:switch of plus and dot if good}} \\
		& = a_1 \oplus ((a_0 \oplus (x \odot b_0)) \odot (x \oplus b_0)) \oplus b_1	&& \by{\cref{l:switch of plus and dot if good}} \\
		& = a_1 \oplus \sigma(a_0, x, b_0) \oplus b_1.
	\end{align*}
	Analogously, $(a_0 \oplus x \oplus b_1) \odot (a_1 \oplus (a_0 \odot x) \oplus b_0) = a_1 \oplus \sigma(a_0, x, b_0) \oplus b_1$.
\end{proof}

\begin{lemma} \label{l:for assoc}
	Let $A$ be an {\mvm}, let $n \in \Np$, let $(a_0, \dots, a_n)$ and $(b_0, \dots, b_n)$ be good sequences in $A$, and let $x \in A$.
	Then
	\[
		\bigodot_{i = 0} ^n a_i \oplus (x \odot b_{n-i-1}) \oplus b_{n-i} = \bigodot_{i = 0}^n a_i \oplus (a_{i-1} \odot x) \oplus b_{n-i}.
	\]
\end{lemma}
\begin{proof}
	We prove the statement by induction on $n \in \Np$.
	The case $n = 1$ is \cref{l:associative 212}.
	Now let $n \in \N\setminus \{0, 1\}$, and suppose that the statement holds for $n-1$.
	In the following chain of equalities, the second equality is obtained by an application of \cref{l:associative 212} with respect to the good pairs $(a_{n-1}, a_n)$ and $(b_0, b_1)$, and the third equality is obtained by an application of the inductive hypothesis with respect to the good sequences $(a_0, \dots, a_{n-1})$ and $(b_1, \dots, b_n)$.
	\begin{align*}
		& \bigodot_{i = 0} ^n a_i \oplus (x \odot b_{n-i-1}) \oplus b_{n-i} \\
		& = \left(\bigodot_{i = 0} ^{n-2} a_i \oplus (x \odot b_{n-i-1}) \oplus b_{n-i} \right) \odot (a_{n-1} \oplus (x \odot b_{0}) \oplus b_1) \\
		& \quad \odot (a_n \oplus x \oplus b_1) \\
		& = \left(\bigodot_{i = 0} ^{n-2} a_i \oplus (x \odot b_{n-i-1}) \oplus b_{n-i} \right) \odot (a_{n-1} \oplus x \oplus b_1) \\
		& \quad \odot (a_n \oplus (a_{n-1} \odot x) \oplus b_1) \\
		& = \left(\bigodot_{i = 0}^n a_i \oplus (a_{i-1} \odot x) \oplus b_{n-i} \right) \odot (a_n \oplus (a_{n-1} \odot x) \oplus b_1) \\
		& = \bigodot_{i = 0}^n a_i \oplus (a_{i-1} \odot x) \oplus b_{n-i}. \qedhere
	\end{align*}
\end{proof}

\begin{lemma} \label{l:associativity}
	Let $A$ be an {\mvm}, let $\gs{a}$ and $\gs{b}$ be good sequences in $A$, and let $x \in A$.
	Then,
	\[
		(\gs{a} + (x)) + \gs{b} = \gs{a} + ((x) + \gs{b}).
	\]
\end{lemma}

\begin{proof}
	Set $\gs{d} \df \gs{a} + (x)$ and $\gs{e} \df(x) + \gs{b}$.
%	We have $d_0 = a_0 \oplus x$ and, f
	For every $n \in \N$, we have $d_n = a_n \oplus (a_{n-1} \odot x)$ and $e_n = (x \odot b_{n-1}) \oplus b_n$.
	We set $\gs{f} \df(\gs{a} + (x)) + \gs{b}$ and $\gs{g} \df \gs{a} + ((x) + \gs{b})$.
	For every $n \in \N$, we have
	\begin{align*}
		f_n	& = \bigodot_{i = 0} ^n d_i \oplus b_{n-i} \\
				& = \bigodot_{i = 0} ^n a_i \oplus (a_{i-1} \odot x) \oplus b_{n-i} \\
				& = \bigodot_{i = 0} ^n a_i \oplus (x \odot b_{n-i-1}) \oplus b_{n-i}	&& \by{\cref{l:for assoc}} \\
				& = \bigodot_{i = 0} ^n a_i \oplus e_{n-i} \\
				& = g_n.																						&& \qedhere
	\end{align*}
\end{proof}

\begin{proposition} \label{p:associativity full}
	Addition of good sequences is associative.
\end{proposition}
\begin{proof}
	By \cref{l:associativity,l:Light,l:generation}.
\end{proof}

Our next aim---reached in \cref{p:distributivity full} below---is to show that good sequences satisfy $\gs{a} + (\gs{b} \lor \gs{c}) = (\gs{a} + \gs{b}) \lor (\gs{a} + \gs{c})$.
We need some lemmas.

\begin{lemma} \label{l:for distributivity}
	Let $A$ be {\amvm}, let $(x_0, x_1)$ and $(y_0, y_1)$ be good pairs in $A$ and let $z \in A$.
	Then 
	\[
		((z \oplus x_1) \odot x_0) \lor ((z \oplus y_1) \odot y_0) = (z \oplus (x_1 \lor y_1)) \odot (x_0 \lor y_0).
	\]
\end{lemma}

\begin{proof}
	We have
	\begin{align} 
		& (z \oplus (x_1 \lor y_1)) \odot (x_0 \lor y_0) \notag\\
		& = ((z \oplus x_1) \lor (z \oplus y_1)) \odot (x_0 \lor y_0) \notag\\
		& = ((z \oplus x_1) \odot (x_0 \lor y_0)) \lor ((z \oplus y_1) \odot (x_0 \lor y_0)) \notag\\
		& = ((z \oplus x_1) \odot x_0) \lor ((z \oplus x_1) \odot y_0) \lor ((z \oplus y_1) \odot x_0) \lor ((z \oplus y_1) \odot y_0) \notag\\
		& = \sigma(z, x_1, x_0) \lor ((z \oplus x_1) \odot y_0) \lor ((z \oplus y_1) \odot x_0) \lor \sigma(z, y_1, y_0), \label{e:distr-lor}
	\end{align}
	where the last equality follows from \cref{l:switch of plus and dot if good}.
	We have
	\begin{align*}
		(z \oplus x_1) \odot y_0	& = 		(z \oplus (x_1 \odot x_0)) \odot y_0					&& \by{$(x_0, x_1)$ is good} \\
											& = 		(z \lor \sigma(z, x_1, x_0)) \odot y_0					&& \by{\cref{ax:A6}} \\
											& = 		(z \odot y_0) \lor (\sigma(z, x_1, x_0) \odot y_0)	&& \by{$\odot$ distr.\ over $ \lor $} \\
											& \leq 	((z \oplus y_1) \odot y_0) \lor \sigma(z, x_1, x_0)	&& \by{\cref{l:order-preserving properties}} \\
											& = 		\sigma(z, y_1, y_0) \lor \sigma(z, x_1, x_0).		&& \by{\cref{l:switch of plus and dot if good}}
	\end{align*}
	Analogously, $(z \oplus y_1) \odot x_0 \leq \sigma(z, x_1, x_0) \lor \sigma(z, y_1, y_0)$.
	Therefore, \eqref{e:distr-lor} equals $\sigma(z, x_1, x_0) \lor \sigma(z, y_1, y_0)$, i.e.\ $((z \oplus x_1) \odot x_0) \lor ((z \oplus y_1) \odot y_0)$.
\end{proof}

\begin{lemma} \label{l:distributivity}
	Let $A$ be an {\mvm}, let $a\in A$ and let $\gs{b}$ and $\gs{c}$ be good sequences in $A$.
	Then, $(a) + (\gs{b} \lor \gs{c}) = ((a) + \gs{b}) \lor ((a) + \gs{c})$.
\end{lemma}

\begin{proof}
	We set $\gs{d} \df (a) + {(\gs{b} \lor \gs{c})}$.
	Then, 
	\[
		d_n = (a \oplus (b_n \lor c_n)) \odot (b_{n-1} \lor c_{n-1}) \odot \dots \odot (b_0 \lor c_0) = (a \oplus (b_n \lor c_n)) \odot (b_{n-1} \lor c_{n-1}).
	\]
	We set $\gs{f} \df (a) + \gs{b}$, $\gs{g} \df (a) + \gs{c}$ and $\gs{h} \df \gs{f} \lor \gs{g} = ((a) + \gs{b}) \lor ((a) + \gs{c})$. We have
	\begin{align*}
		f_n	& = (a \oplus b_n) \odot b_{n-1} \odot \dots \odot b_0 = (a \oplus b_n) \odot b_{n-1}, \\
		g_n	& = (a \oplus c_n) \odot c_{n-1} \odot \dots \odot c_0 = (a \oplus c_n) \odot c_{n-1},
	\intertext{and}
		h_n	& = f_n \lor g_n \\
				& = ((a \oplus b_n) \odot b_{n-1}) \lor ((a \oplus c_n) \odot c_{n-1}) \\
				& = (a \oplus (b_n \lor c_n)) \odot (b_{n-1} \lor c_{n-1})
				&& \by{\cref{l:for distributivity}} \\
				& = d_n.
				&& \qedhere
	\end{align*}
\end{proof}

\begin{proposition} \label{p:distributivity full}
	For all good sequences $\gs{a}, \gs{b}, \gs{c}$ in {\amvm}, we have 
	\begin{align*}
		\gs{a} + (\gs{b} \lor \gs{c}) = (\gs{a} + \gs{b}) \lor (\gs{a} + \gs{c}),
	\intertext{and}
		\gs{a} + (\gs{b} \land \gs{c}) = (\gs{a} + \gs{b}) \land (\gs{a} + \gs{c}).
	\end{align*}
\end{proposition}

\begin{proof}
	Let us prove the first equality: the second one is analogous.
	Let $A$ be the {\mvm} of the statement.
	Set $\hat{A} \df \{\,(x) \in \GS(A) \mid x \in A\,\}$.
	By \cref{l:generation}, $\hat{A}$ generates the magma $\langle \GS(A); + \rangle$.
	Following \cref{n:magma}, for $n \in \Np$, we let $\hat{A}_n$ denote the set of elements of $\GS(A)$ which can be obtained with at most $n$ occurrences of elements of $\hat{A}$ via application of $ + $.
	We prove by induction on $n \in \Np$ that, for all $\gs{a} \in \hat{A}_n$, and $\gs{b}$, $\gs{c} \in \GS(A)$, we have $\gs{a} + (\gs{b} \lor \gs{c}) = (\gs{a} + \gs{b}) \lor (\gs{a} + \gs{c})$.
	The case $n = 1$ is \cref{l:distributivity}.
	Suppose that the statement holds for $n \in \Np$, and let us prove it for $n + 1$.
	Let $\gs{a} \in \hat{A}_{n + 1}$, and let $\gs{b}$, $\gs{c} \in \GS(A)$.
	Then, there exists $\gs{a'} \in \hat{A}_n$ and $\gs{x} \in \hat{A}$ such that $\gs{a} = \gs{a'} + \gs{x}$ or $\gs{a} = \gs{x} + \gs{a'}$.
	Since addition is commutative by \cref{p:sum is commutative}, these two conditions are equivalent.
	So,
	\begin{align*}
		\gs{a} + (\gs{b} \lor \gs{c})	& = \gs{x} + \gs{a'} + (\gs{b} \lor \gs{c}) \\
												& = \gs{x} + ((\gs{a'} + \gs{b}) \lor (\gs{a'} + \gs{c})) && \by{ind.\ hyp.}\\
												& = (\gs{x} + \gs{a'} + \gs{b}) \lor (\gs{x} + \gs{a'} + \gs{c})	&& \by{\cref{l:distributivity}}\\
												& = (\gs{a} + \gs{b}) \lor (\gs{a} + \gs{c}).	&&\qedhere
	\end{align*}
\end{proof}

For a good sequence $\gs{a} = (a_0, a_1, a_2, \dots)$ in an {\mvm}, set $\gs{a} \ominus 1 \df(a_1, a_2, a_3, \dots)$.
The sequence $\gs{a} \ominus 1$ is a good sequence.

\begin{proposition}
	For every {\mvm} $A$, the algebra $\GS(A)$ is {\apulm}.
\end{proposition}

\begin{proof}
	By \cref{p:distributive lattice}, $\GS(A)$ is a distributive lattice.
	By \cref{p:sum is commutative,p:associativity full,r:0 is neutral}, $\GS(A)$ is a commutative monoid.
	By \cref{p:distributivity full}, $ + $ distributes over $ \land $ and $ \lor $.
	Thus, $\GS(A)$ is {\alm} (\cref{ax:P0}).
	Since the order in $\GS(A)$ is pointwise (\cref{r:order is pointwise}), and $0$ is the least element of $A$, we have \cref{ax:P1}, i.e., $\gs{0}$ is the least element of $\GS(A)$.
	It is easy to see that $(a_0, a_1, a_2, \dots) + \gs{1} = (1, a_0, a_1, a_2, \dots)$.
	Therefore, we have \cref{ax:P2}, i.e., for all $\gs{a} \in \GS(A)$, $\gs{a} + \gs{1} \ominus 1 = \gs{a}$.
	For all $\gs{a} \in \GS(A)$, we have $(\gs{a} \ominus 1) + \gs{1} = a \lor 1$, which establishes \cref{ax:P3}.
	By induction, one proves $n \gs{1} = (\underbrace{1, \dots, 1}_{n \text{ times}}, 0, 0, 0 \dots)$.
	Since $1$ is the maximum of $A$, we have \cref{ax:P4}, i.e., for all $\gs{a} \in \GS(A)$, there exists $n \in \N$ such that $\gs{a} \leq n \gs{1}$.
\end{proof}

Given a morphism of {\mvms} $f\colon A\to B$, we set
\begin{align*}
	\GS(f) \colon \GS(A)&\longrightarrow \GS(B) \\
	(x_0, x_1, x_2, \dots)&\longmapsto (f(x_0),f(x_1),f(x_2), \dots).
\end{align*}

\begin{lemma}
	For every morphism $f$ of {\mvms}, the function $\GS(f)$ is a morphism of {\pulms}.
\end{lemma}

\begin{proof}
	Let us prove that $\GS(f)$ preserves $ + $.
	Set $\gs{z} \df \gs{x} + \gs{y}$, $\gs{u} \df f(\gs{z})$, and $\gs{w} \df f(\gs{x}) + f(\gs{y})$.
	Let $\gs{z} = (z_0, z_1, z_2, \dots)$, $\gs{u} = (u_0, u_1, u_2, \dots)$ and $\gs{w} = (w_0, w_1, w_2, \dots)$.
	We shall show $\gs{u} = \gs{w}$.
	For each $n \in \N$, we have 
	\[
		z_n = (x_0 \oplus y_n) \odot \dots \odot (x_n \oplus y_0).
	\]
	Thus,
	\begin{align*}
		u_n	& = f((x_0 \oplus y_n) \odot \dots \odot (x_n \oplus y_0)) \\
				& = (f(x_0) \oplus f(y_n)) \odot \dots \odot (f(x_n) \oplus f(y_0)) \\
				& = f((x_0 \oplus y_n) \odot \dots \odot (x_n \oplus y_0)) \\
				& = w_n.
	\end{align*}
	Therefore, $\GS(f)$ preserves $ + $.
	Straightforward computations show that $\GS(f)$ preserves also $0$, $1$, $ \lor $, $ \land $ and $\ominus$.
\end{proof}

It is easy to see that $\GS\colon \MVM \to \ULMP$ is a functor.

%%%%%%%%%%%%%%%%%%%%%%%%%%%%%%%%%%% SECTION %%%%%%%%%%%%%%%%%%%%%%%%%%%%%%%%%%%%

\section{MV-monoidal algebras and positive cones are equivalent}
\label{s:MVM and positive-unital}

%================================= SUBSECTION =================================%

\subsection{The unit interval functor from positive cones}

Let $M$ be a {\pulm}.
We set $\U(M) \df \{\,x \in M \mid x \leq 1\,\}$; $\U$ stands for `unit interval'.
We endow $M$ with the operations of {\mvm}.
The operations $\lor$, $\land$, $0$, $1$ are defined by restriction.
For $x, y \in \U(M)$, we set $x \oplus y	\df (x + y) \land 1$ and $x \odot y \df (x + y) \ominus 1$.
By the equivalence between $\ULMP$ and $\ULM$ (\cref{t:G_0 is equiv}), and since $\G(M')$ is an {\mvm} for every {\ulm} $M'$ (\cref{p:Gamma is MVM}), $\U(M)$ is an {\mvm}.
Given a morphism $f\colon M \to N$ of {\pulms}, we set $\U(f) \colon \U(M) \to \U(N)$ as the restriction of $f$.
This assignment establishes a functor $\U\colon \ULMP\to\MVM$.

%================================= SUBSECTION =================================%
	
\subsection{The unit}

For each {\mvm} $A$, consider the function
\begin{align*}
	\eta^1_A\colon A&\longrightarrow \U\GS(A) \\
	x&\longmapsto (x).
\end{align*}

\begin{proposition} \label{p:unit mmv}
	For every {\mvm} $A$, the function $\eta^1_A\colon A\to \GS(A)$ is an isomorphism of {\mvms}.
\end{proposition}

\begin{proof}
	The facts that $\eta^1_A$ is a bijection and that it preserves $0$, $1$, $\lor$, $\land$ are immediate.
	Let $x, y \in A$.
	Then, $(x) + (y) = (x \oplus y, x \odot y)$.
	Therefore 
	\begin{align*}
		\eta^1_A(x) \oplus \eta^1_A(y) & = (x) \oplus (y) = ((x) + (y)) \land \gs{1}\\
		& = (x \oplus y, x \odot y) \land \gs{1} = (x \oplus y) = \eta^1_A(x \oplus y), \text{ and}\\
		\eta^1_A(x) \odot \eta^1_A(y) & = (x) \odot (y) = ((x) + (y)) \ominus 1\\
		& = (x \oplus y, x \odot y) \ominus 1 = (x \odot y) = \eta^1_A(x \odot y).\qedhere
	\end{align*}
\end{proof}

\begin{proposition} \label{p:eta1 is natural}
	$\eta^1 \colon \Id_{\MVM} \dot{\to} \U\GS$ is a natural transformation, i.e., for every morphism of {\mvms} $f\colon A\to B$, the following diagram commutes.
	\[
		\begin{tikzcd}
			A \arrow{r}{\eta_A^1} \arrow[swap]{d}{f}	& \U\GS(A) \arrow{d}{\U\GS(f)} \\
			B \arrow[swap]{r}{\eta_B^1}					& \U\GS(B) \\
		\end{tikzcd}
	\]
\end{proposition}

\begin{proof}
	For every $x \in A$ we have
	\[
		\U\GS(f)(\eta_A^1(x)) = \U\GS(f)((x)) = \GS(f)((x)) = (f(x)) = \eta_B^1(f(x)). \qedhere
	\]
\end{proof}

%================================= SUBSECTION =================================%
	
\subsection{The counit}

For each {\pulm}, we consider the function
\begin{align*}
	\eps_M^1 \colon \GS\U(M)&\longrightarrow M \\
	(x_0, \dots, x_n)&\longmapsto x_0 + \dots + x_n.
\end{align*}
Our next goal, met in \cref{p:exists unique good sequence}, is to prove that $\eps_M^1$ is bijective; this will show that a {\pulm} $M$ is in bijection with the set of good sequences in its unit interval $\U(M)$.

\begin{lemma} \label{l:trunc}
	Let $M$ be {\apulm}.
	For every $x \in M$ and every $n \in \N$, we have
	\[
		x = (x \land n) + (x \ominus n).
	\]
\end{lemma}
\begin{proof}
	We have
	\begin{align*}
		(x \land n) + (x \ominus n) + n \stackrel{\text{\cref{l:fixed}}}{ = } (x \land n) + (x \lor n) \stackrel{\text{\cref{l:land+lor=+}}}{ = } x + n.
	\end{align*}
	Since $n$ is cancellative by \cref{l:n is cancellative}, it follows that $(x \land n) + (x \ominus n) = x$.
\end{proof}

\begin{lemma} \label{l:x less than n}
	Let $M$ be {\apulm}, let $x \in M$ and let $n \in \N$. 
	If $x \leq n$, then $x \ominus n = 0$.
\end{lemma}

\begin{proof}
	By \cref{l:fixed}, we have 
	\[
		(x \ominus n) + n = x \lor n = n = 0 + n.
	\]
	By \cref{l:n is cancellative}, the element $n$ is cancellative: it follows that $x \ominus n = 0$.
\end{proof}

\begin{lemma} \label{l:transpose}
	Let $M$ be a {\pulm}, let $x \in M$, and let $n,k\in \N$.
	Then, 
	\[
		(x \ominus n) \land k = (x \land (n + k)) \ominus n.
	\]
\end{lemma}

\begin{proof}
	We have
	\begin{align*}
		((x \ominus n) \land k) + n	& = ((x \ominus n) + n) \land (k + n)	&& \by{$ + $ distr.\ over $ \land $} \\
												& = (x \lor n) \land (k + n)	&&\by{\cref{l:fixed}} \\
												& = (x \land (k + n)) \lor n \\
												& = ((x \land (k + n)) \ominus n) + n.	&& \by{\cref{l:fixed}}
	\end{align*}
	Since $n$ is cancellative by \cref{l:n is cancellative}, we have $(x \ominus n) \land k = (x \land (n + k)) \ominus n$.
\end{proof}

\begin{lemma} \label{l:truncation is good}
	For every $x$ in \apulm, we have
	\[
		(x \land 1) + ((x \ominus 1) \land 1) = x \land 2.
	\]
\end{lemma}

\begin{proof}
	We have
	\begin{align*}
		& (x \land 1) + ((x \ominus 1) \land 1) + 1 \\
		& = (x \land 1) + ((x \ominus 1) + 1) \land 2)				&& \by{$ + $ distr.\ over $ \land $} \\
		& = (x \land 1) + ((x \lor 1) \land 2)							&& \by{\cref{ax:P3}} \\
		& = ((x \land 1) + (x \lor 1)) \land ((x \land 1) + 2)	&& \by{$ + $ distr.\ over $ \land $} \\
		& = (x + 1) \land (x + 2) \land 3								&& \by{\cref{l:land+lor=+}, $ + $ distr.\ over $ \land $} \\
		& = (x + 1) \land 3													&& \\
		& = (x \land 2) + 1.													&& \by{$ + $ distr.\ over $ \land $}
	\end{align*}
	Since $1$ is cancellative by \cref{l:n is cancellative}, we have $(x \land 1) + ((x \ominus 1) \land 1) = x \land 2$.
\end{proof}

\begin{lemma} \label{l:trunc is good}
	Let $M$ be {\apulm}, and let $x \in M$.
	Then, $(x \ominus 0, x \ominus 1, x \ominus 2, \dots)$ is a good sequence in $\U(M)$.
\end{lemma}

\begin{proof}
	For $n \in \N$, set $x_n \df x \ominus n$.
	Since $x \leq n$ for some $n \in \N$, the sequence $(x_0, x_1, x_2, \dots)$ is eventually $0$ by \cref{l:x less than n}.
	We have
	\begin{equation} \label{l:land2}
		x_n + x_{n + 1} = ((x \ominus n) \land 1) + (((x \ominus n) \ominus 1) \land 1) \stackrel{\text{\cref{l:truncation is good}}}{=} (x \ominus n) \land 2.
	\end{equation}
	Therefore,
	\[
		x_n \oplus x_{n + 1} = (x_n + x_{n + 1}) \land 1 \stackrel{\text{\cref{l:land2}}}{=} ((x \ominus n) \land 2) \land 1	 = (x \ominus n) \land 1  = x_n.
	\]
	Moreover,
	\begin{align*}
		(x_n \odot x_{n + 1}) + 1	& = ((x_n + x_{n + 1}) \ominus 1) + 1 \\
										& = (x_n + x_{n + 1}) \lor 1 	&& \by{\cref{ax:P3}}\\
										& = ((x \ominus n) \land 2) \lor 1	&& \by{\cref{l:land2}} \\
										& = ((x \ominus n) \lor 1) \land 2 \\
										& = (((x \ominus n) \ominus 1 ) + 1) \land 2	&&\by{\cref{ax:P3}} \\
										& = ((x \ominus (n + 1)) \land 1) + 1 \\
										& = x_{n + 1} + 1.
	\end{align*}
	The element $1$ is cancellative by \cref{l:n is cancellative}; thus $x_n \odot x_{n + 1} = x_{n + 1}$.
\end{proof}

\begin{lemma} \label{l:trunc sums up}
	For every $m\in \N$ and every element $x$ of {\apulm} such that $x \leq m$, we have
	\[
		x = \sum_{n \in \{0, \dots, m-1\}}(x \ominus n) \land 1.
	\]
\end{lemma}

\begin{proof}
	We prove the statement by induction on $m\in \N$.
	If $m = 0$, then $x = 0$, and the assertion holds.
	Let us suppose that it holds for a fixed $m$, and let us prove that it holds for $m + 1$.
	We recall that, by \cref{l:transpose}, we have $(x \ominus n) \land 1 = (x \land (n + 1)) \ominus n$.
	We have
	\begin{align*}
		x	& = (x \land m) + (x \ominus m)
			&& \by{\cref{l:trunc}} \\
			& = (x \land m) + ((x \land (m + 1)) \ominus n)
			&& \by{$x \leq m + 1$} \\
			& = \left(\sum_{n \in \{0, \dots, m-1\}} ((x \land m) \land (n + 1)) \ominus n \right)\\
			& \quad + ((x \land (m + 1)) \ominus n)
			&& \by{ind.\ hyp.} \\
			& = \left(\sum_{n \in \{0, \dots, m - 1\}}(x \land (n + 1)) \ominus n \right)\\
			& \quad + ((x \land (m + 1)) \ominus n) \\
			& = \sum_{n \in \{0, \dots, m\}} (x \land (n + 1)) \ominus n.
			&& \qedhere
	\end{align*}
\end{proof}

\begin{remark} \label{r:good-unrolling}
	Let $M$ be {\apulm}, and let $x_0, x_1 \in \U(M)$.
	The pair $(x_0, x_1)$ is a good pair in $\U(M)$ if, and only if, $(x_0 + x_1) \land 1 = x_0$ and $(x_0 + x_1) \ominus 1 = x_1$.
	This is just unrolling the definitions.
\end{remark}

\begin{lemma} \label{l:base case}
	Let $M$ be a \pulm. For every $m\in \N$ and every good sequence $(x_0, \dots, x_m)$ in $\U(M)$, we have
	\begin{equation*} \label{e:land}
		(x_0 + \dots + x_m) \land 1 = x_0.
	\end{equation*}
\end{lemma}

\begin{proof}
	We prove the statement by induction on $m\in \N$.
	The case $m = 0$ is trivial.
	The case $m = 1$ holds by \cref{r:good-unrolling}.
	Suppose the statement holds for $m\in \Np$, and let us prove it holds for $m + 1$.
	We have the following chain of equalities, the first of which is justified by the fact that $x_0 + \dots + x_{m-1} + 1 \geq 1$.
	\begin{align*}
		& (x_0 + \dots + x_{m + 1}) \land 1 \\
		& = (x_0 + \dots + x_{m + 1}) \land (x_0 + \dots + x_{m-1} + 1) \land 1 \\
		& = (x_0 + \dots + x_{m-1} + ((x_m + x_{m + 1}) \land 1)) \land 1 && \by{$+$ distr.\ over $\land$}\\
		& = (x_0 + \dots + x_{m-1} + x_m) \land 1 && \by{\cref{r:good-unrolling}}\\
		& = x_0.	&& \by{ind.\ hyp.} \qedhere
	\end{align*}
\end{proof}

\begin{lemma}\label{l:base case lor}
	Let $M$ be {\apulm}. For every $k \in \N$ and every good sequence $(x_0, \dots, x_k)$ in $\U(M)$, we have
	\begin{equation}\label{e:lor m}
		(x_0 + \dots + x_k) \ominus 1 = x_1 + \dots + x_k.
	\end{equation}
\end{lemma}

\begin{proof}
	We prove this statement by induction on $k \in \N$.
	\Cref{e:lor m} is equivalent to 
	\[
		((x_0 + \dots + x_k) \ominus 1) + 1 = 1 + x_1 + \dots + x_k,
	\]
	i.e.,
	\[
		(x_0 + \dots + x_k) \lor 1 = 1 + x_1 + \dots + x_k.
	\]
	The case $k = 0$ is trivial.
	Let us suppose that the statement holds for a fixed $k \in \N$, and let us prove that it holds for $k + 1$.
	We have
	\begin{align*}
		1 + x_1 + \dots + x_{k+1}	& = (1 + x_1 + \dots + x_{k}) + x_{k+1}							&& \\
											& = ((x_0 + \dots + x_{k}) \lor 1) + x_{k+1}						&& \text{(ind. hyp.)}	\\
											& = (x_0 + \dots + x_{k} + x_{k+1}) \lor (1 + x_{k+1})		&& \\
											& = (x_0 + \dots + x_{k+1}) \lor ((x_{k} + x_{k+1}) \lor 1)	&& \\
											& = ((x_0 + \dots + x_{k+1}) \lor(x_{k} + x_{k+1})) \lor 1	&& \\
											& = (x_0 + \dots + x_{k+1}) \lor 1.									&& \qedhere
	\end{align*}
\end{proof}

\begin{lemma} \label{l:unique}
	Let $M$ be a {\pulm}, let $m\in \N$, and let $(x_0, \dots, x_m)$ and $(y_0, \dots, y_m)$ be good sequences in $\U(M)$.
	If 
	\[
		x_0 + \dots + x_m = y_0 + \dots + y_m,
	\]
	then, for all $i \in \{0, \dots,m\}$, $x_i = y_i$.
\end{lemma}

\begin{proof}
	We prove the statement by induction on $m$.
	The case $m = 0$ is trivial.
	Suppose that the statement holds for a fixed $m\in \N$, and let us prove it for $m + 1$.
	By \cref{l:base case}, we have 
	\[
		x_0 = (x_0 + \dots + x_{m + 1}) \land 1 = (y_0 + \dots + y_{m + 1}) \land 1 = y_0.
	\]
	By \cref{l:base case lor}, we have
	\begin{align*}
		x_1 + \dots + x_{m + 1}	& = (x_0 + x_1 + \dots + x_{m + 1}) \ominus 1 \\
										& = (y_0 + y_1 + \dots + y_{m + 1}) \ominus 1 \\
										& = y_1 + \dots + y_{m + 1}.
	\end{align*}
	By inductive hypothesis, for all $i \in \{1, \dots, m + 1\}$, $x_i = y_i$.
\end{proof}

\begin{proposition} \label{p:exists unique good sequence}
	Let $M$ be a \pulm, and let $x \in M$.
	Then, there exists exactly one good sequence $(x_0, \dots, x_m)$ in $\U(M)$ such that $x = x_0 + \dots + x_m$, given by 
	\[
		x_n = (x \ominus n) \land 1.
	\]
	In particular, the function $\eps_M^1 \colon \GS\U(M) \to M$ is bijective.
\end{proposition}

\begin{proof}
	\Cref{l:trunc sums up,l:trunc is good} show that $x_n = (x \ominus n) \land 1$ works.
	Uniqueness is ensured by \cref{l:unique}.
\end{proof}

Our next goal is to prove that $\eps_M^1$ is a morphism of {\pulms} (\cref{p:eps1 is morphism} below).
We need some lemmas.

\begin{lemma} \label{l:join of good sequences in mon}
	Let $M$ be a {\pulm}. For all good sequences $(x_0, \dots, x_m)$ and $(y_0, \dots, y_m)$ in $\U(M)$, we have
	\begin{equation} \label{e:lor}
		(((x_0 + \dots + x_m) \lor (y_0 + \dots + y_m)) \ominus n) \land 1 = x_n \lor y_n,
	\end{equation}
	and
	\begin{equation} \label{e:land m}
		(((x_0 + \dots + x_m) \land (y_0 + \dots + y_m)) \ominus n) \land 1 = x_n \land y_n.
	\end{equation}
\end{lemma}

\begin{proof}
	Let us prove \cref{e:lor}.
	Set $x \df x_0 + \dots + x_n$, and $y\df y_0 + \dots + y_n$.
	By \cref{p:exists unique good sequence}, we have $x_n = (x \ominus n) \land 1$ and $y_n = (y\ominus n) \land 1$.
	Adding $n$ on both sides of \cref{e:lor}, we obtain the equivalent statement
	\[
		(x \lor y \lor n) \land (n + 1) = ((x \lor n) \land (n + 1)) \lor ((y \lor n) \land (n + 1)),
	\]
	which holds by the distributivity laws.
	The proof of \cref{e:land m} is analogous.
\end{proof}

\begin{lemma} \label{l:duo}
	Let $M$ be {\apulm}, and let $x,y \in M$ with $y \leq 1$.
	Set $x_0 \df x \land 1$ and $x_{1} \df (x \ominus 1) \land 1$.
	Then,
	\[
		((x + y) \ominus 1) \land 1 = x_0 \odot (x_1 \oplus y).
	\]
\end{lemma}
\begin{proof}
	We have
	\begin{align*}
		x_0 + x_1	& = (x \land 1) + ((x \ominus 1) \land 1)\\
						& = ((x \land 1) + (x \ominus 1)) \land ((x \land 1) + 1)	&& \by{$+$ distr.\ over $\land$}\\
						& = x \land (x + 1) \land 2	&& \by{\cref{l:trunc}}\\
						& = x \land 2.
	\end{align*}
	Therefore, we have
	\begin{align*}
		x_0 \odot (x_1 \oplus y)	& = (x_0 + ((x_1 + y) \land 1)) \ominus 1	&& \by{def.\ of $\oplus$ and $\odot$}\\
											& = ((x_0 + x_1 + y) \land (x_0 + 1)) \ominus 1 && \by{$+$ distr.\ over $\land$}\\
											& = ((x + y) \land ((x \land 1) + 1)) \ominus 1	&& \\
											& = ((x + y) \land (x + 1) \land 2) \ominus 1	&& \by{$+$ distr.\ over $\land$}\\
											& = ((x + y) \land 2) \ominus 1	&& \\
											& = ((x + y) \ominus 1) \land 1.	&& \by{\cref{l:transpose}} \qedhere
	\end{align*}
\end{proof}

\begin{lemma} \label{l:ominus n split}
	Let $x$ and $y$ be elements of {\apulm}, let $n \in \N\setminus \{0\}$, and suppose $y \leq 1$.
	Then,
	\[
		(x + y) \ominus n = (x \ominus (n-1) + y) \ominus 1.
	\]
\end{lemma}

\begin{proof}
	We have
	\begin{align*}
		& (x \ominus (n-1) + y) \ominus 1 + n	\\
		& = (x \ominus (n-1) + y) \ominus 1 + 1 + (n-1)	&&\\
		& = (x \ominus (n-1) + y) \lor 1 + (n-1)			&&\\
		& = (x \ominus (n-1) + y + (n-1)) \lor n			&&\\
		& = ((x \lor (n-1)) + y) \lor n					&&\\
		& = (x + y) \lor (y + (n-1)) \lor n				&&\\
		& = (x + y) \lor n											&& \by{$y \leq 1$}	\\
		& = ((x + y) \ominus n) + n.							&& \by{\cref{l:fixed}}
	\qedhere
	\end{align*}
\end{proof}

\begin{lemma} \label{l.duo-trans}
	Let $M$ be {\apulm}, and let $x,y \in M$ with $y \leq 1$.
	For every $n \in \Np$, we have 
	\[
		((x + y) \ominus n) \land 1 = ((x \ominus (n-1))\land 1) \odot (((x \ominus n) \land 1) \oplus y).
	\]
\end{lemma}
\begin{proof}
	We have
	\begin{align*}
		& ((x + y) \ominus n) \land 1\\
		& = (x \ominus (n-1) + y) \ominus 1	&& \by{\cref{l:ominus n split}}\\
		& = ((x \ominus (n-1)) \land 1) \odot ((((x \ominus (n-1)) \ominus 1) \land 1) \oplus y)	&& \by{\cref{l:duo}}\\
		& = ((x \ominus (n-1)) \land 1) \odot (((x \ominus n) \land 1) \oplus y).	&& \qedhere\\
%		& = x_{n-1} \odot (x_{n} \oplus y).	&& 
	\end{align*}
\end{proof}

\begin{proposition} \label{p:eps1 is morphism}
	For every {\pulm} $M$, the function $\eps_M^1 \GS\U(M) \to M$ is a morphism of {\pulms}.
\end{proposition}

\begin{proof}	
	Clearly, $\eps_M^1$ preserves $1$.
	Let us prove that $\eps_M^1$ preserves $\lor$.
	Let $\gs{x} = (x_0, \dots, x_m)$ and $\gs{y} = (y_0, \dots, y_m)$ be good sequences in $\U(M)$.
	We shall prove 
	\[
		(x_0 + \dots + x_m) \lor (y_0 + \dots + y_m)  =  (x_0 \lor y_0) + \dots + (x_m \lor y_m).
	\]
	By \cref{p:join is good}, $(x_0\lor y_0,\dots,x_m\lor y_m)$ is a good sequence.
	By \cref{p:exists unique good sequence}, it is enough to show that, for every $n\in \N$,
	\[
		(((x_0 + \dots + x_m) \lor (y_0 + \dots + y_m)) \ominus n) \land 1  =  x_n \lor y_n.
	\]
	This holds by \cref{l:join of good sequences in mon}.
	Analogously, $\eps_M^1$ preserves $\land$.
	
	Let us prove that $\eps_M^1$ preserves $+$.
	We prove, by induction on $n\in \N$, that, for all $m\in \N$, and for all $(x_0,\dots,x_m)$ and $(y_0,\dots,y_n)$ good sequences in $\U(M)$, we have
	\[
		\eps_M^1((x_0, \dots, x_m) + (y_0, \dots, y_n)) = x_0 + \dots + x_m + y_0 + \dots + y_n.
	\]
	
	Let us prove the base case $n=0$.
	Let $m\in \N$, let $(x_0,\dots,x_m)$ be a good sequence in $\U(M)$, and let $y \in \U(M)$.
	Then, from the definition of sum of good sequences, we obtain that	$(x_0,\dots,x_m)+(y)$ is the good sequence $(z_0, \dots, z_{m+1})$ where, for every $k\in \{0, \dots, m+1\}$, $z_k = x_{k-1} \odot (x_k \oplus y)$ (where, by convention, we set $x_{-1}=1$).
	By \cref{l.duo-trans,p:exists unique good sequence}, for every $k\in \N$, we have 
	\[
		((x_0 + \dots + x_m + y) \ominus k) \land 1 = x_{k-1} \odot (x_k \oplus y)  = z_k.
	\]
	By \cref{p:exists unique good sequence}, we have $z_0+\dots+z_{m+1}=x_0+\dots+x_m+y$; this settles the base case.
	
	Let us suppose that the case $n$ holds, for a fixed $n\in\N$, and let us prove the case $n+1$.
	Let $m\in \N$, and let $(x_0,\dots,x_m)$ and $(y_0,\dots,y_{n+1})$ be good sequences in $\U(M)$.
	Then
	\begin{align*}
		& \eps_M^1((x_0, \dots, x_m) + (y_0, \dots, y_{n+1}))	\\
		& = \eps_M^1((x_0, \dots, x_m) + (y_0, \dots, y_n) + (y_{n+1}))	&& \by{\cref{l:as sum of elements}}\\
		& = \eps_M^1((x_0, \dots, x_m) + (y_0, \dots, y_n)) + \eps_M^1((y_{n+1}))
		&& \by{base case}\\
		& = \eps_M^1(x_0, \dots, x_m) + \eps_M^1(y_0, \dots, y_n) + y_{n+1}
		&& \by{ind.\ hyp.}\\
		& = x_0 + \dots + x_m + y_0 + \dots + y_n + y_{n+1}.
		&& \qedhere
	\end{align*}
\end{proof}

\begin{proposition} \label{p:eps1 is natural}
	$\eps^1 \colon \GS\U\dot{\to} \Id_{\ULMP}$ is a natural transformation, i.e., for every morphism of {\pulms} $f\colon M \to N$, the following diagram commutes.
	\[
		\begin{tikzcd}
			\GS\U(M) \arrow[swap]{d}{\GS\U(f)} \arrow{r}{\eps^1_M}	& M \arrow{d}{f} \\
			\GS\U(N) \arrow[swap]{r}{\eps^1_N}								& N
		\end{tikzcd}
	\]
\end{proposition}

\begin{proof}
	Let $(x_0, \dots, x_m) \in \GS\U(M)$.
	Then
	\begin{align*}
		 \eps_N^1(\GS\U(f)(x_0, \dots, x_m))	& = \eps_N^1(\U(f)(x_0), \dots, \U(f)(x_m)) \\
		 													& = \eps_N^1(f(x_0), \dots,f(x_m)) \\
		 													& = f(x_0) + \dots + f(x_m) \\
		 													& = f(x_0 + \dots + x_m) \\
															& = f(\eps_M^1(x_0, \dots, x_m)). \qedhere
	\end{align*}
\end{proof}

%================================= SUBSECTION =================================%

\subsection{The equivalence}

\begin{theorem} \label{t:G_1 is equiv}	
	The functors $\U \colon \ULMP \to \MVM$ and $\GS \colon \MVM \to \ULMP$ are quasi-inverses.
	Thus, the categories of {\ulms} and {\mvms} are equivalent.
\end{theorem}

\begin{proof}
	By \cref{p:unit mmv,p:eta1 is natural}, the two functors $\Id_{\MVM} \colon \MVM \to \MVM$ and $\U\GS\colon \MVM \to \MVM$ are naturally isomorphic.
	By \cref{p:eps1 is morphism,p:eps1 is natural,p:exists unique good sequence}, the functors $\Id_{\ULMP} \colon \ULMP\to \ULMP$ and $\GS\U\colon \ULMP\to \ULMP$ are naturally isomorphic.
\end{proof}

We are ready to prove the main result of the paper.

\begin{theorem} \label{t:G is equivalence}
	The functor $\G\colon \ULM \to \MVM$ is an equivalence of categories.
\end{theorem}

\begin{proof}
	The functor $\G$ is the composite of $(-)^+ $ and $\U$, which are equivalences by \cref{t:G_0 is equiv,t:G_1 is equiv}.
\end{proof}

\noindent Notice that, by \cref{t:G_0 is equiv,t:G_1 is equiv}, a quasi-inverse of $\G$ is given by the composite $\T\GS$.
\[
	\begin{tikzcd}
		\ULM \arrow[yshift = .8ex, bend left]{rr}{\G} \arrow[yshift = .45ex]{r}{(-)^+ }& \ULMP\arrow[yshift = .45ex]{r}{\U} \arrow[yshift = -.45ex]{l}{\T}& \MVM \arrow[yshift = -.45ex]{l}{\GS} \arrow[yshift = -.8ex, bend left]{ll}{\X}.
	\end{tikzcd}
\]
%%
%\begin{remark} \label{r:just-one-step}
%	We have constructed a quasi-inverse of $\G$ as the composite of two functors because in this way the proofs seemed a bit easier to write down.
%	A direct construction for the quasi-inverse of $\G$ maps {\amvm} $A$ to the set of $\Z$-indexed sequences $(\dots, x_{-2}, x_{-.1}, x_0, x_1, x_2, \dots)$ such that $x_n = 1$ for $n$ close enough to $-\infty$, $x_n = 0$ for $n$ close enough to $\infty$, and $(x_n, x_{n + 1})$ is a good pair for every $n \in \Z$.
%\end{remark}
%%

%%%%%%%%%%%%%%%%%%%%%%%%%%%%%%%%%%% SECTION %%%%%%%%%%%%%%%%%%%%%%%%%%%%%%%%%%%%

\section{Further research}	

Some results about commutative $\ell$-monoids in the literature suggest similar ones for algebras in the language $\{\oplus, \odot, \lor, \land, 0, 1\}$.
For example, in \cite{Repnbases} it is shown that the variety generated by $\langle \R; +, \lor, \land \rangle$ does not admit a finite equational basis, and a countable basis is given in the same paper.
Building on these results, the content of the present paper may possibly serve to obtain a nice equational basis for the variety generated by $\langle [0, 1]; \oplus, \odot, \lor, \land, 0, 1 \rangle$ which, we conjecture, is not finitely based; in particular, we conjecture that the variety of $\MVM$-algebras is not generated by $[0, 1]$.

We suspect that, from the results in the present paper, one may deduce a nice axiomatization of the quasi-variety generated by 
$\langle [0, 1]; \oplus, \odot, \lor, \land, 0, 1 \rangle$ and of the class of $\{\oplus, \odot, \lor, \land, 0, 1\}$-subreducts of {\mvs} (conjecturally, these two classes coincide).

%%%%%%%%%%%%%%%%%%%%%%%%%%%%%%%%%% APPENDIX %%%%%%%%%%%%%%%%%%%%%%%%%%%%%%%%%%%%

\appendix

%%%%%%%%%%%%%%%%%%%%%%%%%%%%%%%%%%% SECTION %%%%%%%%%%%%%%%%%%%%%%%%%%%%%%%%%%%%

\section{The equivalence restricts to lattice-ordered groups and MV-algebras}
\label{s:restriction}

In this section, we shortly hint at the fact that Mundici's equivalence follows from our main result.

We recall that a \emph{\extulg} (\emph{\ulg}, for short) is an algebra $\langle G; +, \lor, \land, 0, -, 1 \rangle$ (arities $2$, $2$, $2$, $0$, $1$, $0$) such that $\langle G; \lor, \land \rangle$ is a distributive lattice, $\langle G; +, 0, - \rangle$ is an Abelian group, $ + $ distributes over $ \lor $ and $ \land $, $0 \leq 1$, and, for all $x \in M$, there exists $n \in \N$ such that $x \leq n1$.
We let $\ULG$ denote the category of {\ulgs} with homomorphisms.
For all basic notions and results about lattice-ordered groups, we refer to \cite{BKW}.
In every {\ulg} one defines the constant $-1$ as the additive inverse of $1$.
	
\begin{remark} \label{r:reducts-ulg}
	It is not difficult to prove that the $\{+, \lor, \land, 0, 1, -1\}$-reducts of {\ulgs} are precisely the {\ulms} in which every element has an inverse.
	Moreover, the forgetful functor from $\ULG$ to the category of $\{+, \lor, \land, 0, 1, -1\}$-algebras is full and faithful.
\end{remark}

We recall that an \emph{\mv} $\langle A; \oplus, \lnot, 0 \rangle$ is a set $A$ equipped with a binary operation $\oplus$, a unary operation $\lnot$ and a constant $0$ such that $\langle A; \oplus, 0 \rangle$ is a commutative monoid, $\lnot 0 \oplus x = \lnot 0$, $\lnot\lnot x = x$ and $\lnot(\lnot x \oplus y) \oplus y = \lnot(\lnot y \oplus x) \oplus x$.
We let $\MV$ denote the category of {\mvs} with homomorphisms.
For all basic notions and results about {\mvs} we refer to \cite{Cignoli}.
Via $\oplus$, $\lnot$, $0$, one defines the operations $1 \df \lnot 0$, $x \odot y\df \lnot(\lnot x \oplus \lnot y)$, $x \lor y\df (x \odot \lnot y) \oplus y$, and $x \land y\df x \odot (\lnot x \oplus y)$.

\begin{lemma} \label{l:MV is MVM}
	Given {\amv} $\langle A; \oplus, \lnot, 0 \rangle$, the algebra $\langle A; \oplus, \odot, \lor, \land, 0, 1 \rangle$ is {\amvm}.
\end{lemma}

\begin{proof}
	Since $[0, 1]$ generates the variety of {\mvs} \cite[Theorem~2.3.5]{Cignoli}, it suffices to check that \cref{ax:A1,ax:A2,ax:A3,ax:A4,ax:A5,ax:A6,ax:A7} hold in $[0, 1]$.
	This is the case because $\R$ is easily seen to be a {\ulm} and thus, by \cref{p:Gamma is MVM}, the unit interval $[0, 1]$ is {\amvm}.	
\end{proof}

\begin{remark} \label{r:reducts-ulm}
	Using \cref{l:MV is MVM}, one proves that the $\{\oplus, \odot, \lor, \land, 0, 1\}$-reducts of {\mvs} are the {\mvms} such that, for every $x$, there exists $y$ such that $x \oplus y = 1$ and $x \odot y = 0$.
	Moreover, the forgetful functor from $\MV$ to the category of $\{\oplus, \odot, \lor, \land, 0, 1\}$-algebras is full and faithful.
\end{remark}

Using \cref{r:reducts-ulg,r:reducts-ulm}, it is not too difficult to obtain the following.

\begin{theorem} \label{t:restricts to Mundici}
	The equivalence $\G\colon \ULM \to \MVM$ restricts to an equivalence between $\ULG$ and $\MV$.
\end{theorem}

\begin{remark}
	To establish \cref{t:restricts to Mundici} we used the axiom of choice.
	Precisely, we used the choice-based fact that $[0, 1]$ generates the variety of {\mvs} in order to verify that every {\mv} is {\amvm} (\cref{l:MV is MVM}).
	If one proved without the axiom of choice that the axioms of {\mvms} are satisfied by every MV-algebra (and we suspect this to be possible), one would have a choice-free proof of Mundici's equivalence.
	The properties of lattices, \cref{ax:A2}, the distributivity of $\odot$ over $ \lor $ and the distributivity of $\oplus$ over $ \land $ were part of the original axiomatization of MV-algebras by Chang \cite{Chang1958}, which, as proved in \cite{Mangani1973} (see also \cite[Section~2]{Cignoli_elementary}), is equivalent to the modern one, presented here.	
	A direct proof of \cref{ax:A6,ax:A7} has been obtained with the help of Prover9, but we have not obtained proofs of the distributivity of the lattice, the distributivity of $\oplus$ over $ \lor $, the distributivity of $\odot$ over $ \land $, and \cref{ax:A4,ax:A5}.
\end{remark}

%%%%%%%%%%%%%%%%%%%%%%%%%%%%%%%%%%% SECTION %%%%%%%%%%%%%%%%%%%%%%%%%%%%%%%%%%%%

\section{Subdirectly irreducible MV-monoidal algebras are totally ordered} \label{s:sub-irr_tot-ord}

In this section we prove that every subdirectly irreducible {\mvm} is totally ordered (\cref{t:sub irr is totally ordered}).
We proceed in analogy with \cite[Section~1]{Repn}.
Given {\amvm} $A$, and a lattice congruence $\theta$ on $A$ such that $\lvert A/\theta\rvert = 2$, we set
\[
	\theta^*\df \{\,(a, b) \in A \times A\mid \forall x \in A\ (a \oplus x, b \oplus x) \in \theta\text{ and } (a \odot x, b \odot x) \in \theta\,\};
\]
moreover, with $\0(\theta)$ and $\1(\theta)$ we denote the classes of the lattice congruence $\theta$ corresponding to smallest and greatest elements of the lattice $A/\theta$.
An \emph{MVM-congruence} on an {\mvm} $A$ is an equivalence relation on $A \times A$ that respects $\oplus$, $\odot$, $\lor$, $\land$, $0$, $1$.

\begin{lemma} \label{l:char-theta*}
	Let $A$ be an {\mvm}, and let $\theta$ be any lattice congruence on $A$ such that $\lvert A/\theta\rvert = 2$.
	Then, $\theta^*$ is the greatest MVM-congruence contained in $\theta$.
\end{lemma}

\begin{proof}
	It is not difficult to prove that $\theta^* \seq \theta$ and that $\theta^*$ contains every congruence contained in $\theta$.
	
	We prove that $\theta^*$ is an MVM-congruence.
	The relation $\theta^*$ is an equivalence relation because $\theta$ is so.
	
	In the following, let $a, a', b, b'\in A$, and suppose $(a, a'), (b,b') \in \theta^*$.
	
	Let us prove $(a \lor b, a' \lor b') \in \theta^*$.
	Let $x \in A$.
	Since $(a \oplus x, a' \oplus x) \in \theta$, $(b \oplus x, b' \oplus x) \in \theta$, and $\theta$ is a lattice congruence, we have $((a \oplus x) \lor (b \oplus x), (a' \oplus x) \lor (b' \oplus x)) \in \theta$, i.e., $((a \lor b) \oplus x, (a' \lor b') \oplus x) \in \theta$.
	Analogously, $((a \lor b) \odot x, (a' \lor b') \odot x) \in \theta$.
	This proves $(a \lor b, a' \lor b') \in \theta^*$.	Analogously, $(a \land b, a' \land b') \in \theta^*$.

	Let us prove $(a \oplus b, a' \oplus b') \in \theta^*$.
	Let $x \in A$.
	We shall prove 
	\begin{equation} \label{e:oplus oplus}
		(a \oplus b \oplus x, a' \oplus b' \oplus x) \in \theta,
	\end{equation}
	and
	\begin{equation} \label{e:in theta 0}
		((a \oplus b) \odot x, (a' \oplus b') \odot x) \in \theta.
	\end{equation}

	Since $(a, a') \in \theta^*$, we have $(a \oplus (b \oplus x), a' \oplus (b \oplus x)) \in \theta$.
	Since $(b, b') \in \theta^*$, we have $(b \oplus (a' \oplus x), b' \oplus (a' \oplus x)) \in \theta$.
	Hence, by transitivity of $\theta$, we have $(a \oplus b \oplus x, a' \oplus b' \oplus x) \in \theta$, and so \cref{e:oplus oplus} is proved.
	
	Let us prove \cref{e:in theta 0}.
	By transitivity of $\theta$, it is enough to prove 
	\begin{equation} \label{e:in theta 1}
		((a \oplus b) \odot x, (a' \oplus b) \odot x) \in \theta,
	\end{equation}
	and
	\begin{equation} \label{e:in theta 2}
		((a' \oplus b) \odot x, (a' \oplus b') \odot x) \in \theta.
	\end{equation}
	Let us prove \cref{e:in theta 1}.
	Suppose, by way of contradiction, $((a \oplus b) \odot x, (a' \oplus b) \odot x) \notin \theta$.
	Then, without loss of generality, we may assume $(a \oplus b) \odot x \in \0(\theta)$ and $(a' \oplus b) \odot x \in \1(\theta)$.	We have	$\1(\theta) \ni(a' \oplus b) \odot x \leq x$; thus $x \in \1(\theta)$.
	We have
	\begin{equation*}
		\underbrace{(a \oplus b) \odot x}_{\in \0(\theta)} = \sigma(a, b, x) \land \underbrace{x}_{\in \1(\theta)},
	\end{equation*}
	and thus $\sigma(a, b, x) \in \0(\theta)$.
	We have
	\begin{equation*}
	\underbrace{(a' \oplus b) \odot x}_{\in \1(\theta)} = \sigma(a', b, x) \land \underbrace{x}_{\in \1(\theta)},
	\end{equation*}
	and thus $\sigma(a', b, x) \in \1(\theta)$.
	We have
	\[
		\0(\theta) \ni \sigma(a, b, x) = (a \odot (b \oplus x)) \oplus (b \odot x) \geq a \odot (b \oplus x),
	\]
	and thus $a \odot (b \oplus x) \in \0(\theta)$.
	Since $(a, a') \in \theta^*$, it follows that $a' \odot (b \oplus x) \in \0(\theta)$.
	We have
	\[
		\underbrace{a' \odot (b \oplus x)}_{\in \0(\theta)} = a' \land \underbrace{\sigma(a', b, x)}_{\in \1(\theta)}.
	\]
	Therefore, $a'\in \0(\theta)$.
	We have
	\[
		\1(\theta) \ni \sigma(a', b, x) = (a' \oplus (b \odot x)) \odot (b \oplus x) \leq a' \oplus (b \odot x),
	\]
	and thus $a' \oplus (b \odot x) \in \1(\theta)$.
	Since $(a, a') \in \theta^*$, it follows that $a \oplus (b \odot x) \in \1(\theta)$.
	We have
	\[
		\underbrace{a \oplus (b \odot x)}_{\in \1(\theta)} = a \lor \underbrace{\sigma(a, b, x)}_{\in \0(\theta)}.
	\]
	Therefore, $a\in \1(\theta)$.
	Thus, $a\in \1(\theta)$ and $a'\in \0(\theta)$, and this contradicts $(a, a') \in \theta^*\seq \theta$.
	In conclusion, \cref{e:in theta 1} holds, and analogously for \cref{e:in theta 2}.
	By transitivity of $\theta$, \cref{e:in theta 0} holds.
	Thus, $(a \oplus b, a' \oplus b') \in \theta^*$.
	Analogously, $(a \odot b, a' \odot b') \in \theta^*$.	
\end{proof}

We denote with $\Delta$ the identity relation $\{\,(s,s) \mid s \in A\,\}$.

\begin{lemma} \label{l:existence of lattice congruence}
	If $A$ is a subdirectly irreducible {\mvm}, then there exists a lattice congruence $\theta$ on $A$ such that $\lvert A/\theta\rvert = 2$ and $\theta^* = \Delta$.
\end{lemma}

\begin{proof}
	Since $A$ is distributive as a lattice, it can be decomposed into a subdirect product of two-element lattices.
	Let $\{\theta_i\}_{i\in I}$ be the set of lattice congruences of $A$ corresponding with such a decomposition.
	Then $\bigcap_{i\in I} \theta_i = \Delta$.
	By \cref{l:char-theta*}, each $\theta_i^*$ is an MVM-congruence, and $\Delta\seq\theta^*_i\seq \theta_i$.
	Therefore we have $\bigcap_{i\in I} \theta_i^* = \Delta$, and the fact that $A$ is subdirectly irreducible implies $\theta^*_j = \Delta$ for some $j\in I$.
\end{proof}

\begin{theorem} \label{t:sub irr is totally ordered}
	Every subdirectly irreducible {\mvm} is totally ordered.
\end{theorem}

\begin{proof}
	Let $A$ be a subdirectly irreducible {\mvm}.
	\Cref{l:existence of lattice congruence} entails that there exists a lattice congruence $\theta$ on $A$ such that $\lvert A/\theta\rvert = 2$ and such that $\theta^* = \Delta$, i.e., for all distinct $a, b\in A$, there exists $x \in A$ such that $(a \oplus x, b \oplus x) \notin \theta$ or $(a \odot x, b \odot x) \notin \theta$.
	
	Let $a, b\in A$.
	We shall prove that either $a \leq b$ or $b \leq a$ holds.
	Suppose, by way of contradiction, that this is not the case, i.e., $a \land b\neq a$ and $a \land b\neq b$.
	Since $a \land b\neq a$, there exists $x \in A$ such that $((a \land b) \oplus x, a \oplus x) \notin \theta$ or $((a \land b) \odot x, a \odot x) \notin \theta$.
	Since $a \land b\neq b$, there exists $y \in A$ such that $((a \land b) \oplus y, b \oplus y) \notin \theta$ or $((a \land b) \odot y, b \odot y) \notin \theta$.
	We have four cases.
	\begin{enumerate}
	
		\item \label{i:case oplus oplus}
			Case $((a \land b) \oplus x, a \oplus x) \notin \theta$ and $((a \land b) \oplus y, b \oplus y) \notin \theta$.
			
			\noindent
			Since $a \land b \leq a$ and $a \land b \leq b$, we have $(a \land b) \oplus x \in \0(\theta)$, $a \oplus x \in \1(\theta)$, $(a \land b) \oplus y \in \0(\theta)$, and $b \oplus y \in \1(\theta)$. Then,
			\begin{align*}
				\0(\theta)	& \ni		((a \land b) \oplus x) \lor ((a \land b) \oplus y) \\
								& = 		(a \land b) \oplus (x \lor y)								&& \by{$\oplus$ distr.\ over $ \lor $} \\
								& = 		(a \oplus (x \lor y)) \land (b \oplus (x \lor y))	&& \by{$\oplus$ distr.\ over $ \land $} \\
								& \geq 	(a \oplus x) \land (b \oplus y)							&& \by{\cref{l:order-preserving properties}} \\
								& \in \1(\theta),
			\end{align*}
			which is a contradiction.
			
		\item
			The case $((a \land b) \odot x, a \odot x) \notin \theta$ and $((a \land b) \odot y, b \odot y) \notin \theta$ is analogous to \cref{i:case oplus oplus}.
		
		\item \label{i:case oplus odot}
			Case $((a \land b) \oplus x, a \oplus x) \notin \theta$ and $((a \land b) \odot y, b \odot y) \notin \theta$.
		
			\noindent
			Since $a \land b \leq a$, we have $(a \land b) \oplus x \in \0(\theta)$, and $a \oplus x \in \1(\theta)$. Therefore,
			\[
				\0(\theta) \ni(a \land b) \oplus x = \underbrace{(a \oplus x)}_{\in \1(\theta)} \land (b \oplus x).
			\]
			Hence $b \oplus x \in \0(\theta)$, which implies $b\in \0(\theta)$, which implies $(a \land b) \odot y \in \0(\theta)$ and $b \odot y \in \0(\theta)$, which contradicts $((a \land b) \odot y, b \odot y) \notin \theta$.
		
		\item
			The case $((a \land b) \odot x, a \odot x) \notin \theta$ and $((a \land b) \oplus y, b \oplus y) \notin \theta$ is analogous to \cref{i:case oplus odot}.
			
	\end{enumerate}
	In any case, we are led to a contradiction.
\end{proof}

%%%%%%%%%%%%%%%%%%%%%%%%%%%%%%%%%%% SECTION %%%%%%%%%%%%%%%%%%%%%%%%%%%%%%%%%%%%

\section{Good pairs in subdirectly irreducible MV-monoidal algebras} \label{s:good-pairs_in_sub-irr}

The goal of this section---met in \cref{c:good sequence in sub irr}---is to show that good sequences in a subdirectly irreducible {\mvm} are of the form $(1, \dots, 1, x, 0, 0, \dots)$.

\begin{notation} \label{n:congruences}
	Let $A$ be an {\mvm} and let $x \in A$.
	For $a, a'\in A$, set $ a\sim_\bot^x a'$ if, and only if, there exist $n,m\in \N$ such that 
	\begin{align*}
		a \oplus (\underbrace{x \oplus \dots \oplus x}_{n \text{ times}})	& \geq a',	& a' \oplus (\underbrace{x \oplus \dots \oplus x}_{m\text{ times}})	& \geq a.
	\end{align*}
	Moreover, set $ a\sim_x^\top a'$ if, and only if, there exist $n,m\in \N$ such that 	
	\begin{align*}
		b \odot (\underbrace{x \odot \dots \odot x}_{n \text{ times}})	& \leq b',	& b' \odot (\underbrace{x \odot \dots \odot x}_{m \text{ times}})	& \leq b.
	\end{align*}
\end{notation}

It is not difficult to prove the following.

\begin{lemma} \label{l:are-congr}
	For every {\mvm} $A$ and every $x \in A$, the relation $\sim_\bot ^x$ is the smallest MVM-congruence $\sim$ on $A$ such that $x \sim 0$, and the relation $\sim_x^\top$ is the smallest MVM-congruence $\sim$ on $A$ such that $x \sim 1$.
\end{lemma}

\begin{lemma} \label{l:good approximation for leq}
	Let $A$ be an {\mvm}, let $(x_0, x_1)$ be a good pair in $A$, and let $a, b\in A$ be such that $a \leq b \oplus x_1$ and $a \odot x_0 \leq b$.
	Then, $a \leq b$.
\end{lemma}

\begin{proof}
	Let us first deal with the case $b \leq a$; under this hypothesis, we shall prove $a = b$.
	Since $a \leq b \oplus x_1$, we have 
	\[
		a \odot x_0 \leq (b \oplus x_1) \odot x_0 \stackrel{\text{\cref{l:switch of plus and dot if good}}}{ = } \sigma(b, x_0, x_1).
	\]
	Since $a \odot x_0 \leq \sigma(b, x_0, x_1)$ and $a \odot x_0 \leq b$, we have 
	\[
		a \odot x_0 \leq b \land \sigma(b, x_0, x_1) = b \odot (x_0 \oplus x_1) = b \odot x_0.
	\]
	Since $b \leq a$, we have $b \odot x_0 \leq a \odot x_0$.
	Hence, $a \odot x_0 = b \odot x_0$.
	Analogously, $a \oplus x_1 = b \oplus x_1$.
	Hence
	\[
		\sigma(a, x_0, x_1) \stackrel{\text{\cref{l:switch of plus and dot if good}}}{ = } (a \odot x_0) \oplus x_1 = (b \odot x_0) \oplus x_1 \stackrel{\text{\cref{l:switch of plus and dot if good}}}{ = } \sigma(b, x_0, x_1).
	\]
	Set $s \df \sigma(a, x_0, x_1) = \sigma(b, x_0, x_1)$.
	To prove $a = b$ it is enough to prove $a \lor s = b \lor s$ and $a \land s = b \land s$.
	We have
	\[
		a \lor s = a \oplus (x_0 \odot x_1) = a \oplus x_1 = b \oplus x_1 = b \oplus (x_0 \odot x_1) = b \lor s,
	\]
	and
	\[
		a \land s = a \odot (x_0 \oplus x_1) = a \odot x_0 = b \odot x_0 = b \odot (x_0 \oplus x_1) = b \land s.
	\]
	Hence, $a = b$.

	If we do not assume $b \leq a$, we may replace $b$ with $a \land b$, because
	\[
		a \leq (a \oplus x_1) \land (b \oplus x_1) = (a \land b) \oplus x_1,
	\]
	and $a \odot x_0 \leq a \land b$.
	Since $a \land b \leq a$, by the previous part we have $a \land b = a$, i.e., $a \leq b$.
\end{proof}

\begin{theorem} \label{t:good pairs in sub irr}
	Let $A$ be a subdirectly irreducible {\mvm}.
	Then, for all $x, y \in A$, either $x \oplus y = 1$ or $x \odot y = 0$.
\end{theorem}

\begin{proof}
	Set $u\df x \oplus y$ and $v\df x \odot y$.
	We claim ${\sim_\bot^v} \cap {\sim_u^\top} = \Delta$.
	Indeed, let $a, b\in A$ with $a\sim_\bot^v b$ and $a\sim_u^\top b$.
	Then, there exist $n,m\in \N$ such that
	\begin{align*}
		a	& \leq b \oplus (\underbrace{v \oplus \dots \oplus v}_{m\text{ times}}),	& a \odot (\underbrace{u \odot \dots \odot u}_{n \text{ times}})	& \leq b.
	\end{align*}
	Since $(u,v)$ is a good pair by \cref{l:oplus odot is a good pair}, then, by \cref{l:bipartite}, also
	\[
		(u \odot \dots \odot u,v \oplus \dots \oplus v)
	\]
	is so.
	By \cref{l:good approximation for leq}, $a \leq b$.
	Analogously, $b \leq a$, and therefore $a = b$.
	This settles the claim ${\sim_\bot^v} \cap {\sim_u^\top} = \Delta$.
	By \cref{l:are-congr}, $\sim_\bot^{x \odot y}$ and $\sim_{x \oplus y}^\top$ are MVM-congruences, $x \odot y\sim_\bot^{x \odot y}0$ and $x \oplus y\sim_{x \oplus y}^\top 1$.
	Since $A$ is subdirectly irreducible, either $\sim_\bot^v = \Delta$ or $\sim_u^\top = \Delta$.
	In the former case we have $v = 0$, i.e.\ $x \odot y = 0$; in the latter one we have $u = 1$, i.e.\ $x \oplus y = 1$.
\end{proof}
\begin{corollary}
	Let $(x_0, x_1)$ be a good pair in a subdirectly irreducible {\mvm}.
	Then, either $x_0 = 1$ or $x_1 = 0$.
\end{corollary}
\begin{corollary} \label{c:good sequence in sub irr}
	Every good sequence in a subdirectly irreducible {\mvm} is of the form $(1, \dots, 1, x, 0, 0, \dots)$.
\end{corollary}

\section{Independence of the axioms}

With the help of Mace4 \cite{prover9-mace4}, we verified that, once one of the equivalent \cref{ax:A4,ax:A5} is removed, the axioms of {\mvms} are independent.
In particular, each of the following properties does not follow from the conjunction of the other ones.

\begin{enumerate}

	\item
		$ \lor $ and $ \land $ satisfy the axioms of a lattice.
		
	\item
		$ \lor $ left- and right-distributes over $ \land $ and $ \land $ left- and right-distributes over $ \lor $.
		
	\item 
		$\oplus$ is associative.
		
	\item 
		$\odot$ is associative. 
		
	\item 
		$\oplus$ is commutative. %x \oplus y = y \oplus x$.
	
	\item 
		$\odot$ is commutative. %x \odot y = y \odot x$.
		
	\item 
		$x \oplus 0 = x$ and $0 \oplus x = x$.
		
	\item 
		$x \odot 1 = x$ and $1 \odot x = x$.
		
	\item 
		$\oplus$ left- and right-distributes over $ \lor $. %$(x \lor y) \oplus z = (x \oplus z) \lor (y \oplus z)$ and $x \oplus (y \lor z) = (x \oplus y) \lor (x \oplus z)$.
	
	\item 
		$\odot$ left- and right-distributes over $ \land $. %$(x \land y) \odot z = (x \odot z) \land (y \odot z)$ and $x \odot (y \land z) = (x \odot y) \land (x \odot z)$.
		
	\item 
		$\oplus$ left- and right-distributes over $ \land $. %$(x \land y) \oplus z = (x \oplus z) \land (y \oplus z)$ and $x \oplus (y \land z) = (x \oplus y) \land (x \oplus z)$.
	
	\item 
		$\odot$ left- and right-distributes over $ \lor $. %$(x \lor y) \odot z = (x \odot z) \lor (y \odot z)$ and $x \odot (y \lor z) = (x \odot y) \lor (x \odot z)$.
	
	\item 
		$(x \odot y) \oplus ((x \oplus y) \odot z) = (x \oplus (y \odot z)) \odot (y \oplus z)$ and\\
	$(x \oplus y) \odot ((x \odot y) \oplus z) = (x \odot (y \oplus z)) \oplus (y \odot z)$.
	
	\item 
		$(x \odot y) \oplus z = ((x \oplus y) \odot ((x \odot y) \oplus z)) \lor z$ and\\
		$x \oplus (y \odot z) = x \lor ((x \oplus (y \odot z)) \odot (y \oplus z))$.
	\item	
		$(x \oplus y) \odot z = ((x \odot y) \oplus ((x \oplus y) \odot z)) \land z$ and\\
		$x \odot (y \oplus z) = x \land ((x \odot (y \oplus z)) \oplus (y \odot z))$.
\end{enumerate}
%

%============================== ACKNOWLEDGEMENTS ==============================%

\subsection*{Acknowledgements}

The author is grateful to V.\ Marra for some comments and suggestions.

%%%%%%%%%%%%%%%%%%%%%%%%%%%%%%%   BIBLIOGRAPHY   %%%%%%%%%%%%%%%%%%%%%%%%%%%%%%%

\end{document}